\documentclass[12pt,a4paper]{article}

\usepackage{fullpage}
\usepackage{amsmath,amssymb,amsthm}
\usepackage{mathtools}
\usepackage{etex}
\usepackage{bm}
\usepackage{extramacros}
\usepackage{color}

\newcommand{\barq}{\bar{q}}
\newcommand{\conof}[1]{\operatorname{MG}\left(#1\right)}
\newcommand{\half}{\small{\frac{1}{2}}}
\newcommand{\Px}{\mathbb{P}}
\newcommand{\Rx}{\mathbb{R}}

\newcommand{\Bspaceat}[1]{B_{#1}}
\newcommand{\Cal}[1]{\mathcal #1}
\newcommand{\Differential}[2]{\Deriv#1\ #2}
\newcommand{\Lexp}[1]{L^{\Phi}\left(#1\right)}
\newcommand{\LlogL}[1]{L^{\Psi}\left(#1\right)}
\newcommand{\dBspaceat}[1]{B_{#1}^1}
\newcommand{\dexpbundleat}[1]{S\dmaxexpat{#1}}
\newcommand{\dmaxexpat}[1]{\mathcal E_1 \left(#1\right)}

\newcommand{\dmixbundleat}[1]{\prescript{*}{}S\dmaxexpat{#1}}
\newcommand{\dsdomainat}[1]{\sdomain_{#1}^1}
\newcommand{\expbundleat}[1]{S\maxexpat{#1}}
\newcommand{\expbundle}{S\expmod}

\newcommand{\maxexpat}[1]{\mathcal E \left(#1\right)}
\newcommand{\maxexp}{\mathcal E}
\newcommand{\mixbundleat}[1]{\prescript{*}{}S\maxexpat{#1}}
\newcommand{\mixbundle}{\prescript{*}{}S\expmod}
\newcommand{\nDifferential}[3]{\Deriv^{#1}#2\ #3}
\newcommand{\preBspaceat}[1]{\prescript{*}{}B_{#1}}
\newcommand{\predBspaceat}[1]{\prescript{*}{}B_{#1}^1}
\newcommand{\sdomainat}[1]{\sdomain_{#1}}
\newcommand{\sdomain}{\mathcal S}
\newcommand{\EF}[1]{\operatorname{EF}\left(#1\right)}
\newcommand{\MF}[1]{\operatorname{MF}\left(#1\right)}
\newcommand{\MG}[1]{\operatorname{MG}\left(#1\right)}

\newtheorem{theorem}{Theorem}[section]
\newtheorem{definition}[theorem]{Definition}
\newtheorem{proposition}[theorem]{Proposition}

\newtheorem{example}[theorem]{Example}
\newtheorem{remark}[theorem]{Remark}

\date{\hspace{1cm} \\ First version: 16 January 2016. This version: 26 Oct 2016.\\ 
{\normalsize This paper has been inspired by our 1996 preprint {\tt{arXiv:0901.1308[math.PR] }}}}

\begin{document}

\title{{\normalsize Updated version to appear in: Nielsen, F., Critchley, F., \& Dodson, K.  (Eds), Computational Information Geometry for Image and Signal Processing, Springer, 2016} \\ \hspace{3cm} \\ 
{\bf \large Projection based dimensionality reduction for measure valued evolution equations in statistical manifolds}}
\author{Damiano Brigo \\ Dept. of Mathematics \\ Imperial College London \\ 180 Queen's Gate \\  {\small \tt{damiano.brigo@imperial.ac.uk}}\and   Giovanni Pistone \\de Castro Statistics \\ Collegio Carlo Alberto \\ Via Real Collegio 30 \\ 10024 Moncalieri, IT }

\maketitle

\thispagestyle{empty}

\begin{abstract}

We propose a dimensionality reduction method for infinite--dimensional measure--valued evolution equations such as the Fokker-Planck partial differential equation or the Kushner-Stratonovich resp. Duncan-Mortensen-Zakai stochastic partial differential equations of nonlinear filtering, with potential applications to signal processing, quantitative finance, heat flows and quantum theory  among many other areas. Our method is based on the projection coming from a duality argument built in the exponential statistical manifold structure developed by G. Pistone and co-authors. The choice of the finite dimensional manifold on which one should project the infinite dimensional equation is crucial, and we propose finite dimensional exponential and mixture families. This same problem had been studied, especially in the context of nonlinear filtering, by D. Brigo and co-authors but the $L^2$ structure on the space of square roots of densities or of densities themselves was used, without taking an infinite dimensional manifold environment space for the equation to be projected. Here we re-examine such works from the exponential statistical manifold point of view, which allows for a deeper geometric understanding of the manifold structures at play. We also show that the projection in the exponential manifold structure is consistent with the Fisher Rao metric and, in case of finite dimensional exponential families, with the assumed density approximation. Further, we show that if the sufficient statistics of the finite dimensional exponential family are chosen among the eigenfunctions of the backward diffusion operator then the statistical-manifold or Fisher-Rao projection provides the maximum likelihood estimator for the Fokker Planck equation solution. We finally try to clarify how the finite dimensional and infinite dimensional terminology for exponential and mixture spaces are related.

%\bigskip
%
%AMS classification codes: 53B25, 53B50, 60G35, 62E17, 62M20, 93E11
%
% 53B25: local submanifolds
% 53B50 applications to physics
% 60G35 applications, filtering, signals
% 62E17   Approximations to distributions (nonasymptotic) 
% 62M20   Prediction [See also 60G25}; filtering
% 93E11 filtering 

\bigskip


\bigskip

\end{abstract}

\tableofcontents

%\newpage

\section{Introduction}
In this paper we propose a dimensionality reduction method for infinite dimensional measure valued evolution equations such as the Fokker-Planck (or forward Kolmogorov) partial differential equation or the Kushner-Stratonovich resp. Duncan-Mortensen-Zakai stochastic partial differential equations of nonlinear filtering, with potential applications to signal processing, quantitative finance, physics and quantum theory evolution equations, among many other areas. 

This problem naturally shows up when one has to compute the probability distribution of the solution of a stochastic differential equation, or the conditional probability distribution of the solutions of a stochastic differential equation given a related observation process (filtering). Areas where such problems originate naturally are given in signal processing and stochastic filtering in particular, in quantitative finance, in heat flows, in quantum theory and potentially many others, as we discuss in Section \ref{sec:ideq} below. 

Our method is based on the projection coming from a duality argument built in the non-parametric infinite-dimensional exponential statistical manifold structure developed by G. Pistone and co-authors, whose rich history is summarized in Section \ref{sec:infogeom}. 

Dimensionality reduction and finite dimensional approximations will be based on projection on subspaces, so that the study of subspaces is fundamental. We first consider general subspaces in Section \ref{sec:submodels}, trying also to clarify non-parametric exponential and mixture subspaces, and then move to finite dimensional subspaces in Section \ref{sec:fdsub}. 

Clearly the choice of the finite dimensional manifold on which one should project the infinite dimensional equation is crucial, and we propose finite dimensional exponential and mixture families. This same problem had been studied, especially in the context of nonlinear filtering, by D. Brigo and co-authors. In those works the $L^2$ structure on the space of square roots of densities (based on the map $p \mapsto \sqrt{p}$, leading to the Hellinger distance) or of densities themselves (based on the map $p \mapsto p$, leading to the $L^2$ direct metric) was used, and no infinite dimensional manifold environment space for the equation to be projected was introduced. In fact, the main difficulty here is the fact the cone $L^2_+$ has empty relative interior unless the sample space is finite. Here we re-examine such works when adopting the  exponential statistical manifold as an infinite dimensional environment, which allows for a deeper  understanding of the geometric structures at play. We will see earlier in Section \ref{sec:infogeom} that the statistical manifold approach and the Hellinger approach lead to the same metric in the finite dimensional manifold, whereas the $L^2$ direct approach leads to a different metric. This different ``direct metric" works well with a specific type of finite dimensional mixture families, but since the direct metric structure is not compatible with the finite dimensional metric induced by the statistical manifold we will not pursue it further here but leave it for further work. 

Going back to Section \ref{sec:fdsub}, in that section we further clarify how the finite dimensional and infinite dimensional terminology for exponential and especially mixture spaces are related. In the case of mixtures, one has to be careful in distinguishing mixtures generated by convex combinations of given distributions and sets of distributions that are closed under convex mixing. 

Section \ref{sec:proj} considers the finite dimensional projected differential equation for the approximated evolution in a number of cases, in particular the heat equation and the Fokker-Planck equation, and shows how this is derived in detail under the statistical manifold structure introduced earlier. For the particular case of the Fokker-Planck equation we discuss the interpretation of the projected, finite dimensional law as law of a different process, thus providing a tool for designing stochastic differential equations whose solutions densities evolve in a given finite dimensional family. 
We also discuss how one can measure the goodness of the approximation, show that projection in the statistical manifold structure is equivalent with the assumed density approximation for exponential families, and finally prove that if the sufficient statistics of the exponential family are chosen among the backward diffusion operator eigenfunctions then the projected equation provides the maximum likelihood estimator of the Fokker Planck equation solution.  

Section \ref{sec:conc} concludes the paper, hinting at further research problems. 

This paper is a substantial update of our 1996 preprint \cite{2009arXiv0901.1308B}.

\section{Infinite dimensional measure valued evolution equations}\label{sec:ideq}

Stochastic Differential Equations (SDEs) are used in many areas of mathematics, physics,  engineering and social sciences. SDEs represent extensions of ordinary differential equations to systems that are perturbed by random noise. In many problems, and we will see two important examples below, it is important to characterize the evolution in time of the probability law of the solution $X_t$ of the SDE. This probability law, whose density is denoted usually by $p_t$, satisfies typically a partial differential equation (PDE) called Fokker-Planck (or forward Kolmogorov) equation or a stochastic partial differential equation (SPDE) called Kushner-Stratonovich (or Duncan-Mortensen-Zakai in an unnormalized version) equation, depending on the problem. Such measure-valued evolution equations are typically infinite dimensional, in that their solution curves in time $t \mapsto p_t$ do not stay in an a-priori given finite-dimensional parametric family, or in a finite dimensional manifold, unless very special conditions are satisfied. This implies that PDEs and SPDEs cannot be reduced exactly to ODEs or SDEs respectively, but that finite dimensional approximations of these equations need to be considered. One way to obtain finite dimensional approximations is choosing a finite dimensional subspace of the space where the equations for $p_t$ are written, and project the original PDE or SPDE for $p_t$ onto the subspace, using suitable geometric structures, thus obtaining a finite dimensional approximation that is driven by the best local approximation of the relevant vector fields. In this paper our aim is to clarify what kind of geometric structures can make the above approach fully rigorous. Most past works on dimensionality reduction of measure valued equations, see for example \cite{hanzon87,brigoieee,brigobernoulli,armstrongbrigomcss} to name a few, use the $L^2$ space as a framework to implement the above projection. Here we will use the statistical manifold developed by G. Pistone and co-authors instead.

\subsection{The Fokker-Planck or forward Kolmogorov Equation}

Let us start our formal analysis by introducing the complete probability space $(\Omega,{\cal F},\Px)$, with a filtration $\{{\cal F}_t, \ t\ge 0\}$, on which we 
consider a stochastic process $\{X_t, t \ge 0\}$ of diffusion type, solution of a SDE  in $\Rx^N$.
%adapted to a filtration  $\{{\cal F}_{t}, t \ge 0 \}$.
Let the SDE describing $X$ be
of the following form
\begin{eqnarray}\label{eq:dXSDE}
   dX_t =  f_t(X_t) dt + \sigma_t(X_t) d W_t,
\end{eqnarray}
where $\{W_t, t\ge 0\}$ is an $M$-dimensional standard Brownian motion independent
of the initial condition $X_0$, and the drift $f_t$ and diffusion coefficient $\sigma_t$ are respectively an $N$-dimensional vector function and an $N\times M$ matrix function. We define $a(x) := \sigma_t(x) \sigma_t(x)'$ the $N \times N$ diffusion matrix, where the prime symbol denotes transposition. 
In the following to contain notation we will often neglect the time argument in $f_t$ and $a_t$.
The equation above is an It{\^o} stochastic differential
equation.
%In the following derivation,
%we treat the scalar case.
The following set of assumptions will be in force throughout
the paper.

\begin{itemize}
\item[(A)] Initial condition:~ We assume that the initial
state $X_0$ is independent of the process $W$ and has a density $p_0$
w.r.t.\ the Lebesgue measure on $\Rx^n$,
with finite moments of any order, and 
with $p_0$ almost surely positive.
%
%and we make the following assumptions
%on the coefficients $f_t$, $a_t := \sigma_t^2\,$:
%\begin{itemize}
   \item[(B)] Local strong existence:~$f \in C^{1,0}$, $a \in  C^{2,0}$,
   which means that $f$ is once continuously differentiable
   wrt $x$ and continuous wrt  $t$
   and $a$ is twice continuously differentiable
   wrt $x$ and continuous wrt $t$.
   These assumptions imply in
   particular local Lipschitz continuity.

   \item[(C)] Growth / Non--explosion~:
there exists $K > 0$ such that
\begin{eqnarray*}
  2 x' f_t(x) + \| a_t(x) \| \leq K\, (1+\vert x \vert^2),
\end{eqnarray*}
for all $t\geq 0$, and for all $x\in \Rx^N$.
\end{itemize}

Under assumptions~(A),~(B) and~(C) $\exists !$ solution $\{X_t\,,\,t\geq 0\}$ to the state equation,
see \cite{StroVar}, Theorem 10.2.1.
% with $\phi(x,t) = x^2$.

%Under additional assumptions on the coefficients
%(see e.g. \cite{FrieA} or \cite{StroVarh})
\begin{itemize}
\item[(D)]
We assume that the law of $X_t$ is absolutely continuous and
its density
$p_t(x)$ at $x$ has regularity $C^{2,1}$ in $(x,t)$ and satisfies
the Fokker-Planck equation (FPE):
\begin{eqnarray} \label{eq:FP}
\frac{\partial p_t}{\partial t} = {\cal L}_t^\ast p_t,
\end{eqnarray}
where the backward diffusion operator ${\cal L}_t$
is defined by
\begin{displaymath}
   {\cal L}_t = \sum_{i=1}^N f_i \,
   \frac{\partial}{\partial x_i} + \frac{1}{2}
  \sum_{i,j=1}^N a_{i,j} \frac{\partial^2}{\partial x_i \partial x_j},
\end{displaymath}
and its dual (forward) operator is given by
\begin{displaymath}
   {\cal L}^\ast_t p = -
   \sum_{i=1}^N \frac{\partial}{\partial x_i} (f_i p) + \frac{1}{2} 
 \sum_{i,j=1}^N  \frac{\partial^2}{\partial x_i \partial x_j} (a_{i,j} p).
\end{displaymath}
We assume also $p_t(x)$ to be positive for all $t \ge 0$ and
almost all $x \in \Rx^N$.
\end{itemize}
Assumption (D) holds for example under conditions given by
boundedness of the coefficients $f$ and $a$ plus
uniform ellipticity of $a_t$, see \cite{StroVar} Theorem 9.1.9.
Different conditions are also given in \cite{FrieA},
Theorem 6.4.7.
 
Situations where knowledge of the Fokker-Planck solution is important occur for example in signal processing and quantitative finance, among many other fields. Consider the following two examples.

\subsection{Stochastic Filtering with discrete time observations}
In a filtering problem with discrete time observations, the SDE above \eqref{eq:dXSDE} for $X$ is an unobserved signal, of which we observe in discrete time a function $h$ perturbed by noise, namely a process
\[ Y_{t_k} = h(X_{t_k}) + V_{t_k} \]
where $t_0=0, t_1, \ldots, t_k, \ldots$ are discrete times at which observations $Y$ arrive. The process $V$ is a second Brownian motion, independent of the process $W$ driving the signal $X$, and models the noise that perturbs our observation $h$. 
The filtering problem consists of estimating $X_{t_k}$ given observations $Y_{t_0},Y_{t_1},\ldots,Y_{t_k}$ for all $k =1,2,\ldots$. 
It was shown in \cite{brigobernoulli}, Section 6.2, that one can find a suitable finite dimensional exponential family (including the observation function $h$ among the exponent functions) such that the correction step (Bayes formula) at each arrival of new information is exact.  What really brings about the infinite dimensional nature of the problem is the prediction step: between observation, the density of the signal evolves according to the FPE for $X$, and it is this FPE, and the operator ${\cal L}^\ast$ in particular, that leads to infinite dimensionality. Therefore, to study infinite dimensionality in filtering problems with discrete time observations, it suffices to study the Fokker-Planck equation, see again \cite{brigobernoulli} Section 6.2 for the details. 

\subsection{Filtering with continuous time observations and quantum physics}
Consider again the filtering problem, but assume now that observations arrive in continuous time and are given by a stochastic process
\[ d Y_t=h(X_t) dt + dV_t .\]
In this case the solution of the filtering problem is no longer a PDE but a SPDE driven by the observation process $dY$. The SPDE features the same operator ${\cal L}^\ast$ as the FPE and is infinite dimensional. The SPDE exists in a normalized or unnormalized form, and has been studied extensively. It has been shown that even for toy systems like the cubic sensor ($N=1, M=1, f_t = 0, \sigma_t = 1, h(x) = x^3$) the SPDE solution is infinite dimensional \cite{hazewinkel}.  
Finite dimensional approximations based on finite dimensional exponential and mixture families, building on the $L^2$ structure on the space of densities or their square roots to build a projection, have been considered in \cite{hanzon87,brigobernoulli,brigoieee,armstrongbrigomcss}. Nonlinear filtering equations are not of interest merely in signal processing. Several authors have noticed analogies between the filtering SPDEs hinted at above and the evolution equations in quantum physics, see for example \cite{mitter}. Moreover, the related projection filter developed by D. Brigo and co-authors has been applied to quantum electrodynamics, see for example \cite{handel}.

The SPDE case driven by rough paths such as $dY$ is of particular interest because it combines the geometry in the state space for $X$ and $Y$ and the geometry in the space of probability measures associated with $X$ conditional on $Y$'s history. In this paper we are focusing on the latter but in presence of SPDEs one may have to work with the former as well.  One of the problems in this case is choosing the right type of projection also from the state space geometry point of view and see how the optimality of the SPDE projected solution compares with the local  optimality in the projection of the separate drift and diffusion coefficient vector fields of the SPDE. This is related to the different projections suggested in J. Armstrong and D. Brigo    \cite{armstrongbrigoicms} for evolution equations driven by rough paths. For such equations there is more than one possible projection, depending on the notion of optimality one chooses, which is related to the rough paths properties.

\subsection{Valuation of securities with volatility smile in Mathematical Finance}
In Mathematical Finance, often one models stochastic local volatility for a given asset price $S$ via a two-dimensional SDE under the pricing measure
\begin{eqnarray}\label{eq:stochvol}
dS_t &=& r S_t dt + \sqrt{\xi}_t v(S_t) d W_t\\ \nonumber
d \xi_t &=& k(\theta - \xi_t) dt + \eta \sqrt{\xi_t} d V_t\\ \nonumber   
&&d\langle W, V \rangle_t = \rho\ dt 
\end{eqnarray}
where $r, k, \theta, \eta$ are positive constants, $\rho \in [-1,1]$, and $v$ is a regular function. In case $v(S) =S$ one has the Heston model, whereas for more general $v$'s one has a stochastic-local volatility model. One may also extend the model with a third stochastic process for the short rate $r$, introducing a stochastic process $r_t$ of diffusion type replacing the constant risk free rate $r$, obtaining a three dimensional diffusion. We assume below $r$ is constant.   

To calibrate the model one has to fit a number of vanilla options. To do this, it is important to know the distribution of $S_T$ at different maturities $T>0$. In general, this can be deduced by the solution $p_t$ of the FPE for the two-dimensional diffusion process $X_t = [S_t,\xi_t]'$ by integrating with respect to the second component. However, the solution of the FPE for this $X$ is not know in general and is infinite dimensional. It may therefore be important to be able to find a good finite dimensional approximation for this density in order to value vanilla options in a way that leads to an easier calibration process. 

\bigskip

\subsection{The anisotropic heat equation in physics}
\label{sec:running}
We have mentioned earlier that the $L^2$ structure has been used in the past to project infinite-dimensional measure valued evolution equations for densities $t \mapsto q_t$. This structure has been invoked with the maps $q \mapsto \sqrt{q}$ \cite{brigobernoulli,brigoieee} or even $q \mapsto q$ \cite{armstrongbrigomcss,brigol2}, as we will explain more in detail below. It should be noted that the approach $q \mapsto q$ corresponds  to the classical variational approach to parabolic equations, see e.g. the textbook by H. Brezis \cite[Ch. 8--10]{MR2759829}. A typical example of such approach is the equation whose weak form is
\begin{equation}\label{eq:heath}
 \derivby t \int p_t(x) f(x) \ dx + \int \sum_{ij} a_{ij}(x) \left(\partiald {x_i} p_t(x) \right)\left(\partiald {x_j} f(x)\right) \ dx = 0,
\end{equation}
 where both the density $p_t$ and the test function $f$ belong to a Sobolev's space. This corresponds to the operator's form $\partiald t p_t = {\cal L}^\ast p_t$, with
 \begin{equation*}
   {\cal L}^\ast p(x) = \sum_{ij} \partiald {x_j}  \left(a_{ij}(x) \partiald {x_i} p(x)\right) .
 \end{equation*}

This special case is the heat equation in the anisotropic case when the specific heat is constant, and is an important example of infinite dimensional evolution equation we aim at approximating with a finite dimensional evolution. We will keep this equation as an ongoing working example, and we will refer to it as our {\emph{running example}} throughout the paper.
 
Going back to \eqref{eq:heath}, in the following we will discuss an extension of the exponential statistical bundle to the case where the densities are (weakly) differentiable and belong to a weighted Sobolev's space, see \cite[Sec. 6]{lods|pistone:2015}.    

\bigskip

All the above examples from signal processing in engineering, from social sciences, from physics and quantum physics should be enough to motivate the study of finite dimensional approximations of the FPE or of the filtering SPDE. We will tackle the FPE in the following sections, but many other applications are possible. 

We now move to introduce the environment space where the above equations will be examined, the nonparametric infinite dimensional exponential statistical manifold of Giovanni Pistone and co-authors.

\section{Information geometric background}\label{sec:infogeom}
In this section we review the construction of Information Geometry (IG) via the exponential statistical manifold, as originally developed by G. Pistone and C. Sempi \cite{pistone|sempi:95}. More precisely, we will refer to an updated version of the theory we call (exponential) statistical bundle. Among other applications, we will include a qualification intended to deal with the special case of differentiable densities on a real space where we take a Gaussian probability density as background measure $\mu$. This is referred to shortly as Gaussian space. 

\subsection{The exponential statistical manifold and the $L^2$ approach}\label{sec:expstatL2}

We start with an introduction and we shall move to formal definitions below in Sec. \ref{sec:modelspaces}. This approach to IG considers the space of all positive densities of a measured sample space $(X, \Cal X, \mu)$ which are (in an information-theoretic sense) near a given positive density $p$. The idea is representing each element $q$ of this space with the chart 
\begin{equation}
\label{eq:charts}
s_p \colon q \mapsto \log \frac q p - \expectat p {\log \frac q p} = \log \frac q p + \KL p q.
\end{equation}

We define Banach spaces denoted $B_p$ and domains $\maxexp$ and $\sdomainat p$, such that the mappings $s_p \colon \maxexp \to \sdomainat p \subset \Bspaceat q$, $p \in \maxexp$, defined in  Equation \eqref{eq:charts}, form the \emph{affine} atlas of a manifold modeled on the Banach spaces $\Bspaceat p$, $p \in \maxexp$. An atlas is affine if all change-of-chart transformation are affine functions. The Banach space $\Bspaceat p$, the domain $\maxexp$, and the domain $\sdomainat p$ are formally defined below in Section \ref{sec:modelspaces}. We shall show a crucial property of the model Banach spaces $\Bspaceat p$, $p \in \maxexp$, namely they are all isomorphic to each other.

Each $B_p$ is a vector space of $p$-centered random variables, so that the patches are easily shown to be of an exponential form, precisely each $s_p^{-1} = \euler_p \colon \sdomainat p \to \maxexp$ is given by 
\begin{equation*}
  \euler_p(u) = \expof{u - K_p(u)} \cdot p, \quad u \in \sdomainat p  \subset B_p,
\end{equation*}
where $K_q(u) =  \log \expectat q {\euler^u}$ will be defined more precisely later on in Definition \ref{def:KpGp}.

The affine manifold so constructed is not a Riemannian manifold as the Banach spaces $\Bspaceat p$ are not Hilbert spaces. Instead, the theory specifies a second set of Banach spaces $\preBspaceat p$, $p \in \maxexp$, in natural duality with the $\Bspaceat p$'s, and a second affine atlas of the form
\begin{equation}
  \label{eq:m-chart}
 \eta_p(q) = \frac qp -1 \in \preBspaceat p, \quad q \in \maxexp, 
\end{equation}
discussed by A. Cena and G. Pistone \cite{MR2396032}.

 The result is a non parametric version of S.-i. Amari's IG, see \cite{amari:87dual,amari|nagaoka:2000}. Natural vector bundles based on this (dually) affine Banach manifold can be defined together with the proper parallel transports, leading to a first and second order calculus based on connections derived from such transports. We do not develop this aspect here, see the overview by G. Pistone in \cite{pistone:2013Entropy}.

In the application we consider below, the base space is the Lebesgue space on $\reals^d$ and the reference measure is given by the standard Gaussian density. Recent results allow to qualify the theory by considering densities which are differentiable in the sense of distribution and belong to a particular Sobolev space. This is interesting here because it gives the base to discuss partial differential equations in the variational form, see a few results in B. Lods and G. Pistone \cite{lods|pistone:2015}. 

Many expressions of the density other than Equation \eqref{eq:charts} have place in the literature, for example the use of a deformed logarithm, see e.g. \cite{naudts:2011GTh}. The most classical is the $L^2$-embeddings based on the map $q \mapsto \sqrt q \in L^2(\mu)$ that was used by D. Brigo, B. Hanzon, and F. LeGland in \cite{brigobernoulli,brigoieee} in discussing the approximation of nonlinear filters. This mapping is actually a mapping from the set of densities to the Hilbert manifold of the unit sphere, so that a natural set of charts is given by the charts of the manifold of the unit sphere of $L^2(\mu)$. Viewed as such, this mapping is not a chart, but it can be still used to pull-back the $L^2$ structure in order to project on finite dimensional submanifolds. The relation between the exponential manifold and the $L^2$ unit ball manifold is discussed by P. Gibilisco and G. Pistone \cite{gibilisco|pistone:98},  whereas D. Brigo et al. \cite{brigoieee} view the infinite dimensional evolution equation environment as the whole $L^2$ and so avoid the thorny question of defining an inifinite dimensional manifold structure related to the Hilbert structure. A more refined approach would be either considering an infinite dimensional manifold structure different from the $L^2$ structure, as we do here, or using a moving enveloping manifold for the finite dimensional exponential case \cite{brigobernoulli} from which one can project to the chosen finite dimensional exponential submanifold of densities.
 
In a context quite similar to our own, a new type of chart has been introduced by N. Newton in \cite{MR2948226,MR3126105,MR3356252}, namely $q \mapsto q - 1 + \log q - \expectat \mu {\log q}$. This map is restricted to densities which are in $L^2(\mu)$ and such that $\log p \in L^1(\mu)$. As this domain does not fit well with our exponential manifold, we postpone its study to further research. 

Recently, the larger framework of signed measures has been discussed with applications to Statistics, see N. Ay, J. Jost, H.V. L\^e, and L. Schwachh\"ofer \cite{AJLS:2015:arxiv:1510.07305} and their forthcoming book on \emph{Information Geometry} announced in \cite{AJLS:2015}.

As a further option, the identity representation $q \mapsto q \in L^2(\mu)$ has been shown to be of interest in our problem by J. Armstrong and D. Brigo in \cite{brigol2,armstrongbrigomcss}. This amounts to assuming that densities are square integrable and to using the $L^2$ norm directly for densities, rather than their square roots. This metric is called the ``direct $L^2$ metric" in \cite{armstrongbrigomcss}. The image of this mapping is no longer a subset of the unit sphere in $L^2$, and this has consequences when projecting evolution equations for unnormalized probability densities onto finite dimensional manifolds, in that the projection will not take care of normalization. The identity representation above could possibly be interpreted using the charts $q \mapsto \frac qp - 1$ of Equation \eqref{eq:m-chart} which belongs to $\preBspaceat p \subset L^2_0(p)$, but we do not consider this angle here. We just point out that the direct metric approach leads to a different metric and projection than the exponential statistical manifold, whereas the statistical manifold structure agrees with the $L^2$ Hellinger structure. We will see this explicitly later on in Section \ref{sec:directl2vshellinger}. 

We now proceed to present formal definitions of our approach.

\subsection{Model spaces}
\label{sec:modelspaces}
In a Banach manifold each chart of the atlas takes values in a Banach space. The model Banach spaces need not be equal, but they do need to be isomorphic on each connected component. It is the approach used for example by S. Lang in his textbook \cite{lang:1995}. We begin by recalling our definition of model spaces as introduced first in \cite{pistone|sempi:95} with the purpose of defining a Banach manifold on the set $\pdensities$ of strictly positive densities on a given measure space.

For each $p \in \pdensities$ the Young function $\Phi(x) = \cosh x - 1$ defines the Orlicz spaces $L^{\Phi}(p)$ of random variables $U$ such that $\expectat p {\Phi(\alpha U)} < +\infty$ for an $\alpha > 0$. On Orlicz spaces see for example the monograph by J. Musielak \cite{MR724434}. The vector space $L^\Phi(p)$ is the same as the set of random variables such that, for some $\epsilon > 0$, $\expectat p {\euler^{tU}} < \infty$ if $t \in ]-\epsilon, +\epsilon[$. In other words, the space is characterized by the existence of the moment generating function in a neighborhood of 0. This functional setting is implicit in the classical statistical theory. In fact, parametric exponential families are statistical models of the form
\begin{equation*}
  p(x;\theta) = \expof{\sum_{j=1}^d \theta_j U_j - \kappa(\bm\theta)} \cdot p,
\end{equation*}
where the so-called sufficient statistics $U_j$, $j=1,\dots,d$, necessarily belong to the Orlicz space $L^\Phi(p)$, see e.g. L.D. Brown monograph \cite{brown:86}. We will later adopt the notation $c$ for the sufficient statistics, in line with previous works by Brigo and co-authors on finite dimensional approximations. More generally, given a closed subspace $\mathcal V_p \subset L^\Phi(p)$, a $\mathcal V_p$-exponential family is the set of positive densities of the form $\euler^{U - \kappa(U)} \cdot p$.

We define the subspaces of centered random variables
\begin{equation*}
  \Bspaceat p = L_0^{\Phi}(p) = \setof{U \in L^\Phi(p)}{\expectat p {U}=0}
\end{equation*}
 to be used as model space at the density $p$. The norm of these spaces is the induced Orlicz norm from $L^{\Phi}(p)$. 

A critical issue of this choice of model spaces is the fact the Banach spaces $B_p$ are not reflexive and bounded functions are not dense if the sample space does not consist of a finite number of atoms. Technically, the $\Phi$-function lacks a property called $\Delta_2$ in the literature on Orlicz spaces. Precisely, if $\Psi$ the convex conjugate of $\Phi$, $\Psi(y) = \int_0^y (\Phi')^{-1}(v) \ dv$, $y > 0$, then the Orlicz space $L^\Psi(p)$ is $\Delta_2$, so that it is separable and moreover its dual is identified with $L^\Phi(p)$ in the pairing
\begin{equation*}
  L^\Phi(p) \times L^\Psi(p) \ni (U,V) \mapsto \scalarat p U V = \expectat p {UV}.
\end{equation*}

Moreover, a random variable $U$ belongs to $L^\Phi(p)$ if $\L^\Psi(p) \ni V \mapsto \expectat p {UV}$ is a bounded linear map. We write $\preBspaceat p = L^\Psi_0(p)$ so that there is separating duality $\Bspaceat p \times \preBspaceat p \ni (U,V) \mapsto \expectat p {UV}$. In this duality, the space $\preBspaceat p$ is identified with the elements of the pre-dual of $B_p$ which are centered random variables.

If the sample space is not finite, not all $\Bspaceat p$ are isomorphic, but we have the following crucial result, see G. Pistone and M.-P. Rogantin \cite{MR1704564}, \cite{MR2396032}, M. Santacroce, P. Siri, and B. Trivellato in \cite{santacroce|siri|trivellato:2015}. Before the theorem we need a definition.

\begin{definition}\label{def:KpGp}
\begin{enumerate}
\item For each $p \in \pdensities$, the \emph{moment generating functional} is the positive lower-semi-continuous convex function $G_p \colon B_p \ni U \mapsto \expectat p {\euler^U}$ and the \emph{cumulant  generating functional} is the non-negative lower semicontinuous convex function $K_p = \log G_p$. The interior of the common proper domain $\setof{U}{G_p(U) < +\infty}^\circ = \setof{U}{K_p(U) < \infty}^\circ$ is an open convex set $\sdomainat p$ containing the open unit ball (for the Orlicz norm). 
\item For each $p \in \pdensities$, the \emph{maximal exponential family} at $p$ is
\begin{equation}\label{eq:domaine}
  \maxexpat p = \setof{\euler^{u - K_p(u)}\cdot p}{u \in \sdomainat p}.
\end{equation}
\item
Two densities $p,q \in \pdensities$ are \emph{connected by an open exponential arc}, $p \smile q$, if there exists a one-dimensional exponential family containing both in the interior of the parameters interval. Equivalently, for a neighborhood $I$ of $[0,1]$
\begin{equation*}
 \int_{\Omega} p^{1-t}q^t \ d\mu = \expectat p {\left(\frac qp\right)^t} = \expectat q {\left(\frac pq\right)^{1-t}} < +\infty, \quad t \in I \ .
 \end{equation*}
\end{enumerate}
 \end{definition}

\begin{theorem}[Portmanteau Theorem]
\label{prop:maxexp-pormanteau}
Let $p, q \in \pdensities.$ The following statements are equivalent:
\begin{enumerate}
\item \label{prop:maxexp-pormanteau-1} $p \smile q$ (i.e. $p$ and $q$ are connected by an open exponential arc);
  \item \label{prop:maxexp-pormanteau-2} $q \in \maxexpat p$;
  \item \label{prop:maxexp-pormanteau-3} $\maxexpat p = \maxexpat q$;
  \item \label{prop:maxexp-pormanteau-4} $\log \frac q p \in L^{\Phi}(p) \cap L^{\Phi}(q)$;
  \item \label{prop:maxexp-pormanteau-5} $L^{\Phi}(p)=L^{\Phi}(q)$ (i.e. they both coincide  as vector spaces and their norms are equivalent);
  \item \label{prop:maxexp-pormanteau-6} There exists $\varepsilon >0$ such that $\frac{q}{p} \in L^{1+\varepsilon}(p)$ and $\frac{p}{q} \in L^{1+\varepsilon}(q)$.
\end{enumerate}
\end{theorem}

It follows from this structural result that the manifold we are going to define has connected components which are maximal exponential families. Hence we restrict our study to a given maximal exponential family $\maxexp$, where the mention of a reference density is not required any more.

\subsection{Exponential statistical manifold, statistical bundles}

Let $\maxexp$ be a maximal exponential family. The spaces $\Bspaceat p$, $p \in \maxexp$, are isomorphic under the affine mappings $\etransport p q \Bspaceat p \ni U \mapsto U - \expectat q U \in \Bspaceat q$, $p,q \in \maxexp$ and the pre-dual spaces $\preBspaceat p$, $p \in \maxexp$, are isomorphic under the affine mappings $\mtransport p q \preBspaceat p \ni U \mapsto \frac q p U \in \preBspaceat q$, $p,q \in \maxexp$. Such families of isomorphism are the relevant parallel transports in our construction. Precisely, $\etransport p q$ is the \emph{exponential} transport and $\mtransport p q$ is the \emph{mixture} transport and they are dual semigroups,
\begin{equation*}
  \scalarat q {\etransport p q U}V = \scalarat p U {\mtransport q p V} \quad \text{and} \quad \scalarat q W V = \scalarat p {\etransport p q W}{\mtransport p q V},
\end{equation*}
for $U \in \Bspaceat p$, $V,W\in\Bspaceat q$.

We review below some basic topics from \cite{pistone:2013Entropy} and \cite{lods|pistone:2015}.

\begin{definition}\label{def:bundles}
\begin{enumerate}
\item 
The \emph{exponential manifold} is the maximal exponential family $\maxexp$ with the affine atlas of global charts $(s_p \colon p \in \maxexp)$,
\begin{equation*}
  s_p(q) =\log\frac q p - \expectat p {\log\frac q p}.
\end{equation*}
\item
  The \emph{statistical exponential bundle} $\expbundle$ is the manifold defined on the set 
  \begin{equation*}
    \setof{(p,V)}{p \in \expmod, V \in \Bspaceat p}
  \end{equation*}
 by the affine atlas of global charts 
  \begin{equation*}
    \sigma_p \colon (q,V) \mapsto \left(s_p(q),\etransport q p V\right) \in \Bspaceat p \times \Bspaceat p, \quad p \in \maxexp
  \end{equation*}
\item \label{def:bundles1}
  The \emph{statistical predual bundle} $\mixbundle$ is the manifold defined on the set 
  \begin{equation*}
    \setof{(p,W)}{p \in \expmod, W \in \preBspaceat p}
  \end{equation*}
 by the affine atlas of global charts
  \begin{equation*}
    \prescript{*}{}\sigma_p \colon (q,W) \mapsto \left(s_p(q),\mtransport q p W\right) \in \Bspaceat p \times \preBspaceat p, \quad p \in \maxexp
  \end{equation*}
\end{enumerate}
\end{definition}

It should be noted that the full statistical manifold on positive densities actually splits into connected components which are exponential manifolds $\maxexp$ and that all the charts of  the affine atlases have global domains.

The statistical bundle $\expbundle$ is a specific version of the tangent bundle of the exponential manifold. In fact, if we define $\euler_p = s_p^{-1}$, we have $\euler_p(U) = \euler^{U-K_p(U)}\cdot p$ and for each regular curve $p(t) = \euler^{U(t) - K_p(U(t))} \cdot p$, $U(\cdot) \in C^1(I;B_p)$ the velocity of the expression in the $s_p$ chart is $\dot U(t) \in \Bspaceat p$; viceversa, for each $U \in \Bspaceat p$ we have the regular curve $t \mapsto \euler^{tU-K_p(tU)}\cdot p$.

The general notions of velocity and gradient take a specific form in the statistical bundle. Let $t \mapsto p(t)$ be a regular curve in the exponential manifold and let $f \colon \maxexp \to \reals$ be a regular function.

\begin{definition}
  \begin{enumerate}
  \item The \emph{score} of the curve $t \mapsto p(t)$ is the curve $t \mapsto (p(t),Dp(t)) \in \expbundle$ such that
    \begin{equation*}
      \derivby t \expectat {p(t)} V = \scalarat {p(t)} {V-\expectat {p(t)} V}{Dp(t)} 
    \end{equation*}
for all $V \in L^\Psi(p)$, $p \in \maxexp$.
\item The \emph{statistical gradient} of $f$ is the section $\maxexp \ni p \mapsto (p,\grad f(p)) \in \mixbundle$ such that for each regular curve
  \begin{equation*}
    \derivby t f(p(t)) = \scalarat {p(t)} {\grad f(p(t))}{Dp(t)} .
  \end{equation*}
  \end{enumerate}
\end{definition}

In most cases we are able to identify the score as $Dp(t) = \frac{\dot p(t)}{p(t)} = \derivby t \log p(t)$.  

We turn now to the regularity properties of the cumulant generating funtion.

\begin{proposition}[Properties of the CGF] \label{prop:CGF} Let $K_p$ be the
cumulant generating functional at $p \in \maxexp$ and let $\sdomainat p$ be the interior of the proper domain.
\begin{enumerate}
\item $K_p \colon \sdomainat p \to \reals$ is 0 at 0, otherwise is strictly positive; it is convex and
infinitely Fr\'echet differentiable. The value at 0 of
the differential of order $n$ in the direction $U_1,\dots,U_n \in \Bspaceat p$ is the value of the $n$-th joint cumulant under $p$ of the random variable $U_1,\dots,U_n$.
\item The value at $U \in \sdomainat p$ of
the differential of order $n$ in the direction $U_1,\dots,U_n \in \Bspaceat p$ is the value of the $n$-th joint cumulant under $q = \euler_p(U) = \euler^{U - K_p(U)} \cdot p$ of the random variable $U_1,\dots,U_n$, namely

 \begin{equation*} \nDifferential n {K_p\left(U\right)} {[U_1,\dots,U_n]}
=\left. \frac{\partial^n}{\partial t_1 \cdots \partial t_n} \log \expectat q {e^{t_1 U_1 + \cdots + t_n U_n}}\right|_{t=0}.
\end{equation*}
\item \label{item:firsttwo} In particular, $\frac q p -1 \in
\preBspaceat p$ and
\begin{align} &\Differential {K_p(U)}[V] = \expectat q
V = \scalarat p {\frac q p - 1}V \label{eq:DK} \\ &\nDifferential 2
{K_p(U)}{[U_1,U_2]} = \covat q {U_1}{U_2} = \scalarat q {\etransport p q U_1} {\etransport p q U_2} \label{eq:D2K} .
\end{align}
\end{enumerate}
\end{proposition}

Equations~\eqref{eq:DK} and~\eqref{eq:D2K} above show that the geometry of the exponential manifold is fully encoded in the cumulant generating function $K_p$.  The relevant abstract structure is called Hessian manifold, cf Hirohiko Shima's monograph~\cite{shima:2007}.

\subsection{Maximal exponential families of Gaussian type}
\label{sec:gauss}
In this section we study the specific case of the statistical manifold whose components allow for including the Gaussian density (the Gaussian space case), or a generalised Gaussian density. The aim is to develop a framework where partial differential equations are naturally defined. 

Let $M$ be the standard Gaussian density (Maxwell density) on the $d$-dimensional real space. The maximal exponential family $\maxexpat M$ has special features that we review below from \cite[Sec. 4 and 6]{lods|pistone:2015}. Note that in that reference the Young functions $\Phi$ and $\Psi=\Phi_*$ were explicitly denoted as $(\cosh -1)$ and $(\cosh-1)_*$, respectively.

\begin{proposition} \ 
  \begin{enumerate}
  \item The Orlicz space $L^\Phi(M)$ contains all polynomials of degree up to two.
  \item The Orlicz space $L^{\Psi}(M)$ contains all polynomials.
\item   The entropy $H \colon \maxexpat M \ni p \mapsto - \expectat p {\log p}$ is finite and Frech\'et differentiable with statistical gradient $\grad H(p) = -(\log p + H(p))$.
\end{enumerate}
\end{proposition}

Let us compute the action  on a density $p \in \maxexpat M$ of our running example of partial differential operator in Equation \eqref{eq:heath}, assuming all the needed differentiability. We write $p = \euler^{U - K_M(U)} \cdot M$, $U \in \sdomainat M$, and use repeatedly the equality $\expectat M {f \partiald {x_j} g} = \expectat M {(X_jf -  \partiald {x_j} f)g}$ to get the following:
\begin{equation}\label{eq:runningfirst}
  \partiald {x_j} p(x) = \partiald {x_j} \left(\euler^{U(x) - K_M(U)} M(x)\right) = \left(\partiald {x_j} U(x) - x_j\right) p(x) \ .
\end{equation}

\begin{multline}
\label{runningcomputation}
\partiald {x_i}\left(a_{ij}(x)\partiald {x_j} p(x) \right)= 
\partiald {x_i}\left(a_{ij}(x) \left(\partiald {x_j} {U(x)} - x_j\right) p(x)\right)  = \\ 
\partiald {x_i}\left[a_{ij}(x)\left(\partiald {x_j} U(x) - x_j\right)\right] p(x) + \\ a_{ij}(x)\left(\partiald {x_i} U(x) - x_i\right)\left(\partiald {x_j} U(x) - x_j\right) p(x)
\end{multline} 
and
\begin{multline*}
  p^{-1}(x) \sum_{i,j} \partiald {x_i}\left(a_{ij}(x)\partiald {x_j} p(x)\right) = \\ \sum_{i,j} \partiald {x_i}\left[a_{ij}(x)\left(\partiald {x_j} U(x) - x_j\right)\right] + \\ \sum_{i,j} a_{ij}(x)\left(\partiald {x_i} U(x) - x_i\right)\left(\partiald {x_j} U(x) - x_j\right).
\end{multline*}

Note that the left hand side is a random variable whose expectation at $p = \euler^{U-K_M(U)} \cdot M$ is zero. Hence the right hand side is a candidate to be the expression in a chart of a section of the statistical predual bundle of Definition~\ref{def:bundles}(\ref{def:bundles1}).

\begin{example} 
If $[a_{ij}] = I$, then the expression of the PDE is
\begin{equation*}
  \partiald t U(x,t) = \Delta U(x) - d + \avalof{\nabla U(x) - x}^2,
\end{equation*}
and for $d=1$
\begin{equation*}
   \partiald t U(x,t) =  U''(x) - 1 + (U'(x) - x)^2.
\end{equation*}
This provides a simple example of finite dimensionality. Assume there is a solution of the form $U(x,t) = \theta_0(t) + \theta_1(t)x + \theta_2(t)x^2$, that is $p(x,t)$ is Gaussian. It follows
\begin{multline*}
  U''(x) - 1 + (U'(x) - x)^2 = \\ 2 \theta_2(t) + (\theta_1(t) + 2 \theta_2(t) x - x)^2 = \\ (\theta_1(t)^2 + 2 \theta_2(t)) + 2 \theta_1(t)(2 \theta_2(t) - 1)x + (2\theta_2(t)-1)^2x^2 
\end{multline*}
where the value of the constant $\theta_0(t)$ follows from the section condition $\expectat {p(t)} {U(t)} = 0$.
\end{example}

In the one-dimensional case $d=1$, we can generalize easily the density $M(x)$ to $M_{1,m}(x)$, with $m$ positive even integer, defined as
\begin{equation}\label{eq:generalisedgaussian}
  M_{1,m}(x) \propto \exp\left( -\frac{1}{m} x^m\right).
\end{equation}
We could keep the multivariate case but the combinatorial complexity would become quite challenging, so we explain our idea in the scalar case.

The density $M_{1,m}$, chosen as background density, allows one to have in the exponent of the densities monomial terms up to $x^{m-1}$ without any integrability problem, or up to $x^m$ with restriction on the parameters. Suppose, for example, that we need a family of densities flexible enough to include bimodal densities. A natural choice (see \cite{brigobernoulli,armstrongbrigomcss}) would be $m = 4$ and an exponential family of densities
\[ \propto \exp(\theta_1 x + \theta_2 x^2 + \theta_3 x^3 + \theta_4 x^4 ) \]
with parameters $\theta \in \Theta$, open conved domain. However, if $\theta_4$ goes to zero or even positive then we are in troubles. To avoid this, we may choose as background density $M_{1,6}$, so that
\[ \propto \exp(\theta_1 x + \theta_2 x^2 + \theta_3 x^3 + \theta_4 x^4 ) M_{1,6}(x) =
\exp(\theta_1 x + \theta_2 x^2 + \theta_3 x^3 + \theta_4 x^4 - (1/6) x^6)
 \]
will be always well defined as a probability density, for all $\theta$. We briefly mention that  densities such as the above have a number of computational advantages when used to obtain finite dimensional approximations of infinite dimensional evolution equations such as Fokker-Planck or Kushner-Stratonovich or Zakai. These advantages are related to an algebraic ring structure, see \cite{armstrongbrigomcss}.

Let us discuss the action of differential operators of interest on a density $p \in \maxexpat M_{1,m}$, assuming moreover the differentiability where needed. Dropping the index $(1,m)$ from $M$ for brevity, we write $p = \euler^{u - K_{M}(u)} \cdot M$, $u \in \sdomainat M$, to get the following 
\begin{equation*}
  \partiald {x} p(x) = \partiald {x} \left(\euler^{u(x) - K_M(u)} M(x)\right) = \left(\partiald {x} u(x) - x^{m-1}\right) p(x) \ .
\end{equation*}

\begin{multline*}
\partiald {x}\left( a(x)\partiald {x} p(x) \right) = 
\partiald {x}\left(a(x) \left(\partiald {x} {u(x)} - x^{m-1}\right) p(x)\right)  = \\ 
=\partiald {x}\left[a(x)\left(\partiald {x} u(x) - x^{m-1}\right)\right] p(x) + a(x)\left(\partiald {x} u(x) - x^{m-1}\right)^2 p(x)
\end{multline*} 
and
\begin{multline*}
  p^{-1}(x) \partiald {x}\left(a(x)\partiald {x} p(x)\right) = 
  \partiald {x}\left[a(x)\left(\partiald {x} u(x) - x^{m-1}\right)\right]  + \\ a(x)\left(\partiald {x} u(x) - x^{m-1}\right)^2 
\end{multline*}

\begin{example}
If $a = 1$, which in case $d=1$ is usually obtained from a general diffusion via the Lamperti transform, then the previous equation becomes
\begin{equation*}
  p^{-1}(x) \Delta p(x) = \Delta u(x) - (m-1) x^{m-2} + \avalof{\nabla u(x) - x^{m-1}}^2
\end{equation*}
\end{example}

An important feature of the statistical bundles $\expbundleat M$ and $\mixbundleat M$ is the possibility to define Orlicz-Sobolev spaces (see e.g. \cite{MR724434}) for the fibers and use this setup in the study of partial differential equations, cf. \cite[\S 6]{lods|pistone:2015}. 

\begin{definition}
\label{def:orlich-sobolev}
\ 
\begin{enumerate}
\item
The exponential Orlicz-Sobolev spaces of $\maxexpat M$ are the vector spaces
\begin{align*}
   W_{\Phi}^1 = \setof{f \in \Lexp M}{\partial_j f \in \Lexp M, j = 1, \dots, d} \\
  W_{\Psi}^1 = \setof{f \in \LlogL M}{\partial_j f \in \LlogL M, j = 1, \dots, d} 
\end{align*}
where $\partial_j$ is the derivative in the sense of distributions. These spaces become Banach spaces when endowed with the graph norm. The spaces defined with respect to any $p \in \maxexpat M$ are equal as vector space and isomorphic as Banach spaces.  
\item The $W_{\Phi}^1$-exponential family at $M$ is
  \begin{equation*}
    \dmaxexpat M = \setof{\euler^{u - K_M(U)} \cdot M}{U \in \sdomainat M \cap W_{\Phi}^1}
  \end{equation*}
The set $\dsdomainat M = \sdomainat M \cap W_{\Phi}^1$ is a convex open set 

\begin{equation*}
\dsdomainat M \subset \dBspaceat M = \setof{U \in W_{\Phi}^1}{\expectat M U = 0}  
\end{equation*}
 It contains all coordinate functions $X_i$  and polynomials of order two, cf \cite{pistone:2014SIS}.
\end{enumerate}
\end{definition}

The following proposition shows the regularity of the densities in the $W_{\Phi}^1$-exponential family $\dmaxexpat M$ and the Stein's identity in the Orlicz-Sobolev setup, cf. \cite[\S 6]{lods|pistone:2015}. It should be noted that these properties were actually needed above in the derivation of the expression of the running example of PDE.
\begin{proposition} 
Assume $U \in \dsdomainat M$, $p = \euler^{U - K_M(U)} \cdot M\in \dmaxexpat M$, and $f \in W^{1}_{\Phi}$. 
  \begin{enumerate}
  \item It follows $f \euler^{U-K_p(U)} \in W^{1}_{\Phi_{*}}$ and $f \euler^{U-K_p(U)} \cdot M = fp \in W^{1}_{\Phi_{*}}$.
\item $\nabla \euler^{U-K_p(U)} = \nabla U \euler^{U - K_p(U)}$ and $\nabla (\euler^{U-K_p(U)}M) = (\nabla U - \bX) \euler^{U - K_p(U)} M$.
\item (\emph{Multiplication operator}) If $f \in W_{\Psi}^{1}$, then $X_jf \in \LlogL M$.
\item (\emph{Stein's identity}) If $f \in W_{\Psi}^{1}$ and $g \in W_{\Phi}^1(M)$, then
   \begin{equation*}
      \scalarat M f{\partial_j g} = \scalarat M {X_j f - \partial_j f}{g}.
    \end{equation*}
  \end{enumerate}
\end{proposition}

We now define a differentiable version of the statistical bundles.

\begin{definition}
\label{def:diffbundle}
\begin{enumerate}
\item \label{def:diffbundle1}
  The (statistical) \emph{differentiable exponential bundle} is the manifold defined on the set 
  \begin{equation*}
  \dexpbundleat M  =  \setof{(p,V)}{p \in \dmaxexpat M, V \in \dBspaceat p}
  \end{equation*}
 by the affine atlas of global charts 
  \begin{equation*}
    \sigma_p \colon (q,V) \mapsto \left(s_p(q),\etransport q p V\right) \in \dBspaceat p \times \dBspaceat p, \quad p \in \dmaxexpat M
  \end{equation*}
\item \label{def:diffbundle2}
  The (statistial) \emph{differentiable predual bundle} is the manifold defined on the set of fibers 
  \begin{equation*}
   \dmixbundleat M =  \setof{(p,V)}{p \in \dmaxexpat M, V \in \predBspaceat p}
  \end{equation*}
 by the affine atlas of global charts 
  \begin{equation*}
\prescript{*}{}\sigma_p \colon \dmixbundleat M \ni (q,V) \mapsto \left(s_p(q), \mtransport q p V\right) \in \dBspaceat p \times \predBspaceat p \ ,
   \end{equation*}
\end{enumerate}
\end{definition}

We have given a setup such that we can look at a parabolic equation $\partiald t p(x,t) = \mathcal L p(x,t)$ as the equation $p(x,t)^{-1}\partiald t p(x,t) = p(x,t)^{-1}\mathcal L p(x,t)$, where the left hand side is the score of the solution curve $t \mapsto p(t)$ and the right hand side is a section of an appropriate statistical bundle. This type of equation requires the development of a full theory. We here restrict to finite dimensional cases, where the section is actually a section of a finite dimensional submodel.

\section{Submodels and submanifolds}\label{sec:submodels}

Before turning to the main topic of this paper, namely finite dimensional approximations, requiring finite dimensional subspaces structures to be introduced, we study more general subspaces structures that can still be infinite dimensional in general. In particular, this will lead to a first definition of exponential and mixture families associated to subspaces. We will see that while this general exponential family subspace will be similar to the finite dimensional case we will use for the approximation later, the mixture case is subtler, as there are two different notions of mixture family that may however coincide in special cases.

We first consider the following adaptation of the standard definition of sub-manifold, as it is for example given in the monograph \cite{lang:1995} or that by R. Abraham, J.E. Marsden and T. Ratiu \cite{MR960687}. Our definition is tentative and it is intended to go along with the special features of the exponential manifold $\maxexp$, namely the duality between the pre-fibers $\preBspaceat p$ and the fibers $\Bspaceat p$, $p \in \maxexp$. We shall consider two types of substructure, that we call respectively sub-model and sub-manifold.

\begin{definition}[Sub-model, sub-manifold]\label{def:submani}
Let $\Cal N$ be a subset of the maximal exponential family $\maxexp$ and,
for each density $p \in {\Cal N}$, let $V_p^1$ be a closed subspace of $\Bspaceat p$ and $V_p^2$ a closed subspace of $\preBspaceat p$, such that  $V_p^1 \cap V_p^2 = \set{0}$ with continuous immersions $\Bspaceat p \hookrightarrow V_p^1 \oplus V_p^2 \hookrightarrow \preBspaceat p$. Let $\sigma$ be a diffeomorphism of a neighborhood $\Cal W_p$ of $p$ onto the product of two open sets $\Cal V_p^1 \times \Cal
V_p^2$ of $V_p^1 \times V_p^2$ that maps $\mathcal N \cap \Cal W_p$ onto $\Cal V_p^1 \times \set 0$. Assume there exists an atlas $\Sigma$ of such mappings $\sigma$  that covers $\mathcal N$.
\begin{enumerate}
\item \label{item:submani1}
It follows that $\mathcal N$ is a manifold with charts $\sigma_{|\mathcal N}$, $\sigma \in \Sigma$, with tangent spaces $T_p \mathcal N$ isomorphic to $V_p$, $p \in \mathcal N$. We say that such a manifold is a \emph{sub-model} of $\maxexp$.
\item \label{item:submani2}
If the space $V_p^2$ is a closed subspace of $\Bspaceat p$, that is $V_p^1$ \emph{splits} in $\Bspaceat p$, then $\mathcal N$ is a \emph{sub-manifold} of $\maxexp$.
\end{enumerate}
\end{definition}

It should be noted that the splitting condition in Item 2 above is quite restrictive in our context. In fact, while a closed subspace of an Hilbert space always splits with its orthogonal complement, the same is not generally true in our Orlicz spaces. It is generally true only in the finite state space case. However, in the applications we are looking for, either the space $V_p^1 \subset B_p$ or the space $V_p^2 \in \preBspaceat p$ is finite dimensional. Each one of these assumptions allows for a special treatment, as it is shown in the following sections. 

The submanifold issue was originally discussed in \cite{MR1704564}. In particular, it was observed there that each $p$-conditional expectation provides a splitting in $\Bspaceat p$, because $U \mapsto \condexpectat p {U}{\mathcal Y}$ is an idempotent continuous linear mapping on $\Bspaceat p$. The complementary space is the kernel of the conditional expectation. It follows, for example, that each marginalization is a submersion of the exponential manifolds.

The classical theory of parametric exponential families (see \cite{brown:86}) uses a special splitting of the parameter's space which is called mixed parameterization. Our approach actually mimics the same approach in a more abstract and functional language. In fact, if $V_p^1$ is a closed subset of the space $\Bspaceat p$, its orthogonal space or annihilator is actually a subspace of the predual space $\preBspaceat p$, so that $(V_p^1)^\perp \subset \preBspaceat p$. For this reason we have slightly modified the classical definition of sub-manifold in order to accommodate this special structure of interest.
\subsection{Exponential family and mixture (-closed) family submodels}

Our basic example of sub-model is an exponential family in the maximal exponential family $\maxexp$.

\begin{definition}[Exponential family $\EF{V_p}$] \label{def:EF}
Let $V_p$ be a closed subspace of $\Bspaceat p$ and define
\begin{equation*}
\EF{V_p} = \setof{q \in \maxexpat p}{s_p(q) \in V_p} \ .  
\end{equation*}
That is, each $q \in \EF{V_p}$ is of the form $q = \euler^{u - K_p(u)} \cdot p$ with $u \in V_p \cap \sdomainat p$.
\end{definition}

Recall the exponential transport $\etransport p q \colon \Bspaceat p \to \Bspaceat q$, $p,q \in \maxexp$ is defined by $\etransport p q U = U - \expectat q U$. We define the family of parallel spaces $V_q = \etransport p q V_p$, $q \in \maxexp$. The exponential families of two parallel spaces, $\EF{V_p}$ and $\EF{V_q}$, are either equal or disjoint. If fact, if $q \in \EF{V_p}$ then $q = \expof{\bar U - K_p(\bar U)} \cdot p$ and for each $U \in V_p$ it holds 
\begin{multline*}
  \expof{U - K_q(U)} \cdot p = \expof{U - K_p(U) - \bar U + K_p(\bar U)} \cdot q = \\ \expof{\etransport p q {(U - \bar U)} - \expectat q {U - \bar U} + K_p(U) + K_p(\bar U)} \cdot q = \\ \expof{V - K_q(V)} \cdot q
\end{multline*}
with $V = \etransport p q {\left(U - \bar U\right)} \in V_q$. If $q \notin \EF{V_p}$ then there is no common part otherwise the previous computation would show equality. 

The exponential families based on the transport of a subspace $V_p$ form a partition in a covering of statistical models. The next notion of mixture family provides a way to choose a representative in each class.

The mixture family and the complementary spaces are defined as follows.

\begin{definition}[Mixture-closed family]
\label{def:mixturefamily}
\begin{enumerate}
\item 
For each closed subspace $V_p \subset \Bspaceat p$ define its orthogonal space to be its annihilator $V_p^\perp \subset \preBspaceat p$, that is $V_p^\perp = \setof{v \in \preBspaceat p}{\scalarat p vu = 0, u \in V_p}$.
\item The \emph{mixture-closed family} (or mixture family shortly) of $V_p$, is the set of densities $\MF{V_p} \subset \maxexp$ with zero expectation on $V_p$,
  \begin{equation*}
     \MF{V_p} = \setof{q\in\maxexp}{\expectat q U = 0, U \in V_p} \ .
  \end{equation*}
 Equivalently, the set of its mixture coordinates centered at $p$ belongs to $V^\perp_p$, 
 \begin{equation*}
   \eta_p\left(\MF{V_p}\right) = \setof{\frac qp - 1}{q \in \mathcal M(V_p)} = V_p^\perp \cap \eta_p\left(\maxexp\right).
 \end{equation*}
\end{enumerate}
\end{definition}

\begin{remark}
The mixture family $\MF{V_p}$ is convex and deserves its name because itis closed under mixtures, that is convex combinations. However, this name could be misleading as this set in not closed topologically, since we assumed it to be a subset of the maximal exponential family $\maxexpat p$. In general, our mixture families will not contain any extremal point nor will they be generated by a mixture of extremal points. Hence ``closed'' is to be understood in the convex combination sense and not topologically. We will come back to this distinction in the finite dimensional case below. The general problem of mixtures in a maximal exponential family has been discussed in \cite{santacroce|siri|trivellato:2015}
\end{remark}

As we defined the family of subspaces parallel to $V_p$ to be $V_q = \etransport p q V_p$, $q \in \maxexp$, similarly we have the parallel family of orthogonal spaces $V_q^\perp = \mtransport p q V_p^\perp$, where the mixture transport $\mtransport p q \colon \preBspaceat p \to \preBspaceat q$ is defined by $\mtransport p q V = \frac p q V_p$. In fact, $\scalarat q {\etransport p q U}{V} = \scalarat p U {\mtransport q p V}$. The mixture families $\MF{V_q}$, $q \in \maxexp$, are either equal or disjoint. In fact, if $q \in \MF{V_p}$, then
\begin{multline*}
  \MF{V_q} = \setof{r \in \maxexp}{\expectat r V = 0, V \in V_q} = \setof{r \in \maxexp}{\expectat r {\etransport p q U}=0, U \in V_p} = \\ \setof{r \in \maxexp}{\expectat r U = \expectat q U, U \in V_p} = \MF{V_p} \ .
\end{multline*}

The following proposition clarifies the relative position of $\EF{V_p}$ and $\MF{V_p}$.

\begin{proposition}\label{prop:splitting}
\begin{enumerate}
\item The unique intersection of $\EF{V_p}$ and $\MF{V_p}$ is $p$.
\item The space of scores at $q$ of regular curves in $\EF{V_p}$ is $V_q$. 
\item If a regular curve through  $r$ is contained in $\MF{V_p}$, then its score at $r$ is contained in $V_r^\perp$.
\item Assume $V_p^1$ splits in $\Bspaceat p$ with complementary space $V_p^2$. Then both $\EF{V_p^1}$ and $\EF{V_p^2}$ are sub-manifolds of $\maxexp$ with tangent spaces at $p$ respectively $V_p^2$ and $V_p^2$. 
\item Assume $V_p^1$ splits in $\Bspaceat p$ with complementary space $V_p^2$, $u = \Pi_1(u)+\Pi_2(u)$ and assume the mapping
  \begin{equation*}
    q \mapsto u = s_p(q) \mapsto (\Pi_1(u),(\nabla K_p)^{-1}\circ\Pi_2(u))
  \end{equation*}
is a diffeomorphism around $p$. Then $\MF{V_p}$ is a sub-manifold of $\maxexp$ with tangent spaces at $p$ equal to $V_p^2$. 
\end{enumerate}
\end{proposition}

\begin{proof}
\begin{enumerate}
\item First, $p = \euler^{0-K_p(0)} \cdot p$ and $\expectat p U = 0$, if $U \in V_p$. Second, assume $q \in \EF{V_p} \cap \MF{V_p}$. It follows that $q = \euler^{U - K_p(U)} \cdot p$ and $\expectat q U = 0$ for a $U \in V_p$. Hence $0 \ge \KL q p = \expectat q {U - K_p(U)} = \expectat q U - K_p(U) = -K_p(u) \le 0$, hence $U = 0$ and $q=p$. 
\item Follows easily from the definition of exponential family.
\item For $r(t) = \euler^{u(t)-K_r(u(t))} \cdot r \in \MF{V_r}$ and $u \in V_r^1$ we have 
  \begin{equation*}
  0 =   \left. \derivby t {\expectat {r(t)} u} \right|_{t=0}  =   \left. \covat {r(t)} u {\dot u(t)} \right|_{t=0}  = \scalarat r u {\dot u(0)} \ .
  \end{equation*}
\item Let $\Pi_i$, $i=1,2,\ldots$ be the projections induced by the splitting and let $S$ be the open convex set such that both $u_1, u_2 \in \sdomainat p$, namely $S = \Pi_1^{-1}(\sdomainat p) \cap \Pi_2^{-1}(\sdomainat p)$. The mapping $q \mapsto (u_1,u_2)$ satisfies Definition \ref{def:submani}\eqref{item:submani2}.
\item As the main assumption in Definition \ref{def:submani}\eqref{item:submani2} is now an assumption, we have only to check the image of $U \mapsto (0,(\nabla K_p)^{-1}(U_2))$. In fact, $q = \euler_p{(\nabla K_p)^{-1}(U_2)}$ satisfies
  \begin{multline*}
    \expectat q {V} = dK_p\circ (\nabla K_p)^{-1}(U_2) [V] = \\ \scalarat p {(\nabla K_p)\circ (\nabla K_p)^{-1}(U_2)}{V} = \scalarat p {U_2}{V} = 0, \quad V \in V_p^1 \ . 
  \end{multline*}
%\qed
\end{enumerate}
\end{proof}

\begin{figure}[t]
  \centering
  \includegraphics[scale=1]{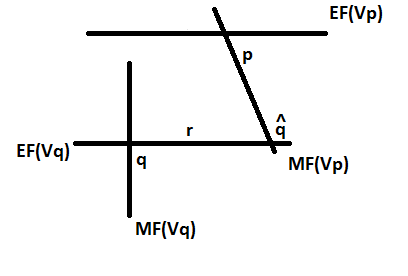}
  \caption{Mixed charts}
  \label{fig:mixedcharts}
\end{figure}
In conclusion, each $p \in \maxexp$ is at the intersection of an exponential and a mixture family and such families can be sub-models or sub-manifolds under proper conditions. This provides a special type of coordinate system namely a mixed system, partly exponential and partly mixture, see Fig. \ref{fig:mixedcharts}. The following proposition summarizes basic facts from the literature and relates the splitting we are looking for with the classical characterization of exponential families, cf e.g. I. Csiszar's paper \cite{csiszar:1975} and the monograph \cite{brown:86}. Special cases of interest will be discussed in the following sections.

\begin{proposition}\label{prop:mixed}
Let be given $p \in \maxexp$ and $V_p \hookrightarrow \Bspaceat p$, so that the families $\EF{V_p}$ and $\MF{V_p}$ are defined. 
\begin{enumerate}
\item Assume that $q \in \maxexp$ is such that the intersection of $\EF{V_q}$ and $\MF{V_p}$ is non empty and contains $\hat q$. The triple of densities $q, \hat q, r$, $r \in \MF{V_p}$ satisfies the \emph{Pythagorean identity} 
    \begin{equation*}
      \KL r {\hat q} + \KL {\hat q} q = \KL r q
    \end{equation*}
and the equivalent equation
\begin{equation*}
  \expectat r {\log\frac{\hat q}q} = \expectat {\hat q} {\log\frac{\hat q}q}
\end{equation*}
\item It follows that any such intersection $\hat q$ strictly minimizes the divergence of $\MF{V_p}$ with respect to $q$, namely
  \begin{equation*}
    \KL {\hat q} q \le \KL r q, \quad r \in \MF{V_p} \ ,
  \end{equation*}
with equality only if $r=\hat q$.
\item Then such intersection $\hat q$ is unique and moreover $\EF{V_q}=\EF{V_{\hat q}}$ and $\MF{V_p} = \MF{V_{\hat q}}$.
\item 
Assume there is an open neighborhood $\mathcal W_p$ of $p\in\maxexp$ such that for each $q \in \mathcal W_p$ there exist the intersection $\hat q = $ of $\EF{V_q}$ and $\MF{V_p}$. We can uniquely write $q = \euler^{\hat u - K_{\hat q}(\hat u)} \cdot \hat q$ with $\hat u \in V_{\hat q}^1$ and $\hat q \in \MF{V_p}$, The map
  \begin{equation*}
   \mathcal W_p \ni q \mapsto (\hat u - \expectat p {\hat u}, \frac {\hat q}p - 1) \in V_p \times V_p^\perp
  \end{equation*}
is injective and separates $\EF{V_p}$ and $\MF{V_p}$.
\end{enumerate}
\end{proposition}

\begin{proof}
  \begin{enumerate}
  \item Let us write $\hat q \in \EF{V_q} \cap \MF{V_p}$ and $r \in \MF{V_p}$ in the chart centered at $q$ as $\hat q = \euler^{\hat u - K_q(\hat u)} \cdot q$ and $r = \euler^{v - K_q(v)} \cdot q$. 
    \begin{multline*}
      \KL r q - \KL r {\hat q} - \KL {\hat q} q = \\ \expectat r {v - K_q(v)} - \expectat r {v - K_q(v) - \hat u + K_q(\hat u)} - \expectat {\hat q} {\hat u - K_q(\hat u)} = \\ \expectat r {\hat u} - \expectat {\hat q} {\hat u} =  \expectat r {\hat u - \expectat p {\hat u}} - \expectat {\hat q} {\hat u - \expectat p {\hat u}} = 0 \ ,
    \end{multline*}
because $\hat u - \expectat p {\hat u} \in V_p$ and both $\hat q, r \in \MF{V_p}$. 
  \item Follows from the Pythagorean Identity and properties of the divergence.
  \item Follows from the previous inequality and the definition of the families. 
  \item Let us write 
    \begin{equation*}
      \log \frac qp = \log \frac q{\hat q} + \log \frac {\hat q}p
    \end{equation*}
with: $q = \euler^{u - K_p(u)} \cdot p$, $u \in \Bspaceat p$; $\hat q = \euler^{v - K_p(v)} \cdot p$, $v \in \Bspaceat p$ and $\expectat {\hat q} v = 0$ if $v \in V_p$; $q = \euler^{\hat u - K_{\hat q}(\hat u)} \cdot \hat q$, $\hat u \in V_{\hat q}$. It follows
\begin{equation*}
  u - K_p(u) = \hat u - K_{\hat q}(\hat u) + v - K_p(v).
\end{equation*}
The $p$-expectation on both sides gives 
\begin{equation*}
  - K_p(u) = \expectat p {\hat u} - K_{\hat q}(\hat u) - K_p(v),
\end{equation*}
so that the equality becomes
\begin{equation*}
  u = \hat u - \expectat p {\hat u} + v.
\end{equation*}
This splitting is unique, because $0 = \hat u - \expectat p {\hat u} + v$ implies $v \in V_p$, hence $
\hat q = \euler^{v - K_{p}(v)} \cdot p \in \EF{V_p} \cap \MF{V_p}$, so that $\hat u = 0$ and $v=0$.

% The Pytagorean identity $\KL q p = \KL q {\hat q} + \KL {\hat q} p$ becomes
% %
%     \begin{equation*}
% \expectat q {u - K_p(u)} = \expectat q {\hat u  - K_{\hat q}(\hat u)} + \expectat {\hat q} {v - K_p(v)}      
%     \end{equation*}
% %
% that is
% %
% \begin{equation*}
%   \expectat q u - \expectat q {\hat u} - \expectat {\hat q} {v} = K_p(u) - K_{\hat q} {\hat u} - K_{p}(v) \ . 
% \end{equation*}
% %
% From the expansion of $q = \euler{u_0 - K_{\hat q}(u_0)} \cdot \hat q$ and the previous equality we obtain
% %
% \begin{equation*}
% u - \expectat q u = u_0 - \expectat q {u_0} + u_1 -\expectat {\hat q} {u_1} \ . 
% \end{equation*}
% %
% It follows
% %
% \begin{align*}
%   & (u - \expectat q u) = (u_0 - \expectat q {u_0}) + (u_1 - \expectat q {u_1}) \\
% & \expectat q {u_1} = \expectat {\hat q} {u_1}
% \end{align*}
  \end{enumerate}
\end{proof}

\section{Finite dimensional families}\label{sec:fdsub}

The most important practical applications of dimensionality reduction for infinite dimensional problems aim at transforming an infinite dimensional problem into a finite dimensional one. This is because, in order to be able to implement a numerical method in a machine, one needs a finite dimensional approximation. It is therefore particularly important to study finite dimensional submanifolds of the statistical manifold on which we might wish to approximate the full, infinite dimensional solution of a problem.

\subsection{Finite dimensional exponential family EF$(c)$}
\label{sec:finiteexp}
Our first special case is the parametric exponential family associated to a finite family of random variable $c = (c_1,\dots,c_n)$. 
\begin{eqnarray}\label{def:finitedimexp}
  \EF c &=& \set{p(\cdot,\theta), \theta \in \Theta},\\ \nonumber
  p(\cdot,\theta) &=& p_\theta = \expof{\theta^T c(\cdot) - \psi(\theta)},
\end{eqnarray}
where $\Theta$ is a maximal convex open set in $\reals^n$.

From the definition it is clear that all densities in the exponential family are connected by an open exponential arc. It follows that the exponential family is a subset of the maximal exponential family containing any of its elements, say $\maxexp = \maxexpat {p}$, for some $p \in \EF c$. In fact, it is a special case of Definition \ref{def:EF}. Precisely, the expression of each $p_\theta \in \EF c$ in the chart $s_{p}$ is given by 
\begin{align*}
  p(\cdot,\theta) &= \expof{\theta^T c(\cdot) - \psi(\theta)} \\
 &= \expof{(\theta-\theta_0)^T c(\cdot) - (\psi(\theta)-\psi(\theta_0))}  \cdot p\\
 &= \expof{(\theta-\theta_0)^T (c(\cdot) - \expectat {p_0} {c}) - (\psi(\theta)-\psi(\theta_0) - (\theta-\theta_0)^T \expectat {p_0} {c})}  \cdot p \\ &= \expof{U(\theta) - K_{p_0}(U(\theta))} \cdot p
\end{align*}
with
\begin{align*}
  U(\theta) &= (\theta-\theta_0)^T (c(\cdot) - \expectat {p_0} {c}) \in \Bspaceat {p} \\
 K_{p_0}(U(\theta)) &= \psi(\theta)-\psi(\theta_0) - (\theta-\theta_0)^T \expectat {p} {c}
\end{align*}

For each $\theta \in \Theta$ let us define the subspace $V_\theta^1$ of $\Bspaceat {p_{\theta}}$ given by
\begin{align}
  V_{\theta}^1 &= V_{p_\theta}^1 = \spanvar{c_1 - \expectat {p_{\theta}} {c_1},\dots,c_n - \expectat {p_{\theta}} {c_n}} \nonumber  
  \\  &= \spanof{c_j - \partiald {\theta_j} \psi (\theta)}{j = 1,\dots,n} \label{eq:tangspexp}
\end{align}
and let $\Pi_\theta \colon \Bspaceat {p_\theta} \to V_\theta^1$ be the orthogonal projector. The orthogonal projection is well defined because $\Bspaceat {p_\theta} \hookrightarrow L^2_0(p_\theta)$ and $V_p^1$ is a closed subspace of $L^2_0(p_\theta)$. If $g(\theta) = [\covat{p_\theta}{c_i}{c_j}]_{i,j=1}^n = \hessianof \psi(\theta)$ is the Fisher Information matrix of the exponential family and $[g^{ij}]_{i,j=1}^n = g^{-1}(\theta)$ denotes its inverse, then for all $U \in B_{p_\theta}$.
\begin{equation}\label{eq:projexpu}
  \Pi_\theta U = \sum_{j=1}^n \sum_{i=1}^n g^{ij}(\theta) \covat {p_\theta} U {c_i}(c_j - \expectat {p_\theta} {c_j}).
\end{equation}

The mapping
\begin{equation*}
  \Bspaceat {p_\theta} \ni U \mapsto (\Pi_\theta U , (I - \Pi_\theta)U) \in V_\theta^1 \times V_\theta^2 \ ,
\end{equation*}
with
\begin{equation*}
  V_\theta^2 = (I - \Pi_\theta)\Bspaceat {p_\theta} = \setof{V \in \Bspaceat {p_\theta}}{\scalarat {p_\theta} V U = 0, U \in V_\theta^1} \hookrightarrow (V_{p_\theta}^1)^\perp, 
\end{equation*}
is a splitting because the decomposition is unique and the spaces are both closed.

Here Definition \ref{def:submani}\eqref{item:submani2} applies and splitting chart at $p_\theta$ is defined on the open domain where the projection is feasible, namely $\setof{p = \euler_{p_\theta}(U) \in \maxexp}{\Pi_\theta U \in \sdomainat {p_\theta}}$, by 
  \begin{multline*}
    p \mapsto U = s_{p_\theta}(p) \mapsto (U^1 = \Pi_{\theta} U , U^2 = U - \Pi_\theta U) \mapsto (\euler_{p_\theta}(U^1),U^2) \\ \in \operatorname{EF}(c) \times (\sdomainat {p_\theta} \cap \kerof {\Pi_\theta})
  \end{multline*}

Note that this splitting chart does provide an immersion of the exponential family into the maximal exponential family, together with a complementary model given by the infinite dimensional exponential family $\maxexp_{\kerof {\Pi_{\theta}}}(p_{\theta}) = \setof{\euler_{p_\theta}(U^2)}{\Pi_\theta(U^2)=0}$, but it does not provide directly a complementary submanifold in the form of a mixture model. However, a different approach is usually taken to describe the complementary manifold, namely Propositions \ref{prop:splitting} and \ref{prop:mixed}.

Let us fix $p_0 = p_{\theta_0} \in \EF{c}$ with associated vector space of centered statistics $V_0^1 \subset \Bspaceat {p_0}$. Consider the vector space $V_0^2 = \setof{U^2 \in \preBspaceat {p_0}}{\scalarat{p_{\theta_0}}{U^2}{U^1} = 0, U^1 \in V_0^1}$, and observe that the mapping $\eta_{p_0} \colon \sdomainat {p_0} \ni U \mapsto d K_{p_0}(U) \in \Bspaceat {p_0}^*$, defined by $\scalarat {p_0} {V}{\eta_{p_0}(U)} = dK_{p_0}(U) [V]$, $V \in \Bspaceat {p_0}$, is one-to-one because of the strict convexity of the cumulant generating functional $U \mapsto K_{p_0}(U)$.  

Assume now $U \in \sdomainat {p_0}$ and moreover $\eta_{p_0}(U) \in V_0^2$. It follows that the corresponding density $\euler_{p_0}(U) \in \maxexp$ is such that
\begin{equation*}
  \expectat {\euler_{p_0}(U)} {U^1} = dK_{p_0}(U)[U^1] = \scalarat {p_0} {\eta_{p_0}(U)}{U^1} = 0, \quad U^1 \in V_0^1 \ .
\end{equation*}

% Given $U \in \sdomainat {p_0}$, assume there exists $U^1 \in V_0^1 \cap \sdomainat {p_0}$ and $U^2 \in \sdomainat {p_0}$ such that $\expectat {\euler_{p_0}(U^2)} V^1 = 0$, $V_1 \in V_0^1$, and
% %
% \begin{equation*}
%   \euler_{p_0}(U) = \expof{U^2 - K_{\euler_{p_0}(U^1)}(U^2)} \euler_{p_0}(U^1)
% \end{equation*}

Let $\operatorname{EF}(c)=\operatorname{EF}(c_1,\dots,c_n)$ be an exponential family in the maximal exponential family $\maxexp$, and let $V^1 = \spanvar{c_1,\dots,c_n}$. Let us define the linear family
  \begin{equation*}
    \mathcal L(c;\alpha) = \setof{q \in \maxexp}{\expectat q c = \alpha} \ ,  
  \end{equation*}
  where the expected value is meant to be applied componentwise. 
  \begin{proposition}
  \begin{enumerate}
  \item Given $q \in \maxexp$, compute the expected value of the $c$'s statistics, $\expectat q c = \alpha$, so that $q$ belongs to the linear family $\mathcal L(c;\alpha)$. Assume there is a nonempty intersection $p \in \EF c \cap \mathcal L(c;\alpha)$, namely $p \in \EF c$ such that $\expectat p c = \expectat q c$. Then such a $p$ is unique. 
  \item Let us express $q$ in the chart centered at $p$, $q = \euler_p(U^2)$. Then $\eta_p(U^2)$ is orthogonal to $V_p^1$.
  \item $p$ is the {\emph{information-projection}} of any element $\bar p$ of the exponential family $\EF c$ on $\mathcal L (c;\expectat q c)$, that is
    \begin{equation*}
      \KL p {\bar p} \le \KL r {\bar p}, \quad r \in \mathcal L (c;\expectat q c), \bar p \in \EF c \ ,
    \end{equation*}
and the Pytagorean equality holds
\begin{equation*}
  \KL q p + \KL p {\bar p} = \KL q {\bar p}
\end{equation*}
\item \label{item:reverseI} $p$ is the {\emph{reverse information-projection}} of $q$ on the exponential family $\EF c$, that is
  \begin{equation*}
    \KL q p \le \KL q {\bar p}, \quad \bar p \in \EF c, p \in \EF c \cap \mathcal L (c,\expectat q c) \ .
  \end{equation*}
  \end{enumerate}
\end{proposition}

\begin{proof}
    \begin{enumerate}
    \item Follows from the strict convexity of the cumulant generating function $\theta \mapsto \psi(\theta)$ and $\expectat {p_\theta} {c_j} = \partial_j \psi(\theta)$, $j=1,\dots,n$ and $\theta \in \Theta$. If $\partial_j\psi(\theta_1) = \partial_j\psi(\theta_2)$, $j = 1,\dots,n$, then $\sum_{j=1}^n (\partial_j\psi(\theta_1) - \partial_j(\theta_2))(\theta_{1j}-\theta_{2j}) = 0$, which implies $\theta_1=\theta_2$ because of $\nabla \psi$ strict monotonicity.
    \item The defining equality is equivalent to $\expectat q {c_j - \expectat p {c_j}} = 0$, $j=1,\dots,n$, hence $\expectat q V = 0$ if $V \in V_p^1$. It follows $0 = d K_p(U^2) [V] = \scalarat p {U^2} V$.
      \item Let us express $r$ and $\bar p$ in the chart centered at $p$, namely $r = \euler_p(U^2)$ and $\bar p = \euler_p(U^1)$, so that $\expectat r {U^1} = \expectat p {U^1} = 0$. It follows that
        \begin{align*}
          \KL r {\bar p} &- \KL p {\bar p} \\ &= \expectat r {U^2 - U^1 - K_p(U^2) + K_p(U^1)} - \expectat p {- U^1 + K_p(U^1)} \\
&= \expectat r {U^2} - K_p(U^2) \\
&= \KL r p 
        \end{align*}
The Pythagorean equality is proved by expressing each density in the chart centered at $p$.
\item By expressing $\bar p$ in the chart centered at $p$, namely $\bar p = \euler_p(U^1)$, $U^1 \in V_p^1$, we have
  \begin{align*}
    \KL q {\bar p} - \KL q p &= \expectat q {\log\frac q{\bar p}} - \expectat q {\log\frac qp} \\
&= \expectat q {\log\frac p {\bar p}} \\
&= - \expectat q {U^1} + K_p(U^1) = K_p(U^1) 
  \end{align*}
which is minimized at $U^1 = 0$
    \end{enumerate}
  \end{proof}

\begin{remark} \begin{enumerate}
  \item 
For each $q \in \maxexp$ such that there exists $p \in \EF c$ satisfying the previous proposition, there is a splitting parameterization $q \mapsto (p,\euler_p(q)) \in \EF c \times V_p^\perp$. The critical issue is the closure of $V_p^\perp$ into $\Bspaceat p$.
\item
Item \ref{item:reverseI} suggests to characterize the feasible set for the splitting by considering the minimum of the mapping 
\begin{equation*}
  q \mapsto \argmin \setof{\KL q {\bar p}}{\bar p \in \EF c} \ .
\end{equation*}
Let us assume (without restriction) that the entropy $H(q) = - \expectat q {\log q}$ is finite, so that $\KL q {\bar p} = -H(q) + \expectat q {\log \bar p} = -H(q) + \sum_{j=1}^n \theta_j \expectat q {c_j} - \psi(\theta)$. we have
\begin{equation*}
 \inf \KL q {\bar p} = -H(q) + \max \theta'\expectat q c - \psi(\theta) = -H(q) + \psi_*(\expectat q c)
\end{equation*}
It follows that the feasible set for the splitting is the open set
\begin{equation*}
\setof{q \in \maxexp}{ \expectat q c \in \domof{\psi_*}^\circ}
\end{equation*}
\end{enumerate}
\end{remark}

\subsection{Finite dimensional mixture(-generated) family $\conof{ q}$}

The basic splitting we have used in the previous sections consists of a closed subspace $V_p^1 \subset \Bspaceat p$ together with its pre-dual annihilator $V_p^2 \subset \preBspaceat p$. As the model space $\Bspaceat p$ is not an Hilbert space unless the base space is finite, there is no identification of $V_p^1 \times V_p^2$ within $\Bspaceat p$, but we only have the immersion $\Bspaceat p \hookrightarrow V_p^1 \oplus V_p^2$. However, the technicalities are somehow easier to control if one of the two splitting spaces is finite dimensional, as it was the case for $V_p^1$ in the previous section.

We have defined a mixture-closed (by convex combinations) family $\MF{V_p}$ in Definition \ref{def:mixturefamily}. Here, we first define a family as the mixture generated by a given family through convex combinations and later we show how this is related with the mixture-closed family. Suppose we are given $n+1$ fixed probability densities, say $q = [q_1,q_2,\ldots,q_{n+1}]^T$. Consider the convex hull of $q$, generated by all possible convex combinations of $q$ elements, which we term ``mixture generated family'' (MG)
\begin{equation*}
  \conof {q} = \setof{\theta^T  q}{\theta \in \Delta(n)} \ ,
\end{equation*}
were $\Delta(n) = \setof{\theta \in \reals_+^{n+1}}{\sum_{i=1}^{n+1} \theta_i = 1}$ is the standard simplex.

We now state a proposition giving conditions under which the two different notions of mixture family coincide in the finite dimensional case, namely we give conditions under which MF$ = $MG.

\begin{proposition}\ 
\begin{enumerate}
\item 
  If all $q_i$ belong to the same maximal exponential family $\maxexpat p$, then $\conof{ q} \subset \maxexpat p$. In particular, we can choose $p \in \conof{ q}$.
\item In such a case, let $V_p^1 = \setof{U \in \Bspaceat p}{\expectat {q_j} U = 0, j=1,\dots,n+1}$. Then this space is closed in $\Bspaceat p$ and $\conof{q} \subset \MF{V_p^1}$.
\item If moreover $\hat q = \sum_{i=1}^{n=1} \alpha_i q_i$ with $\sum_{i=1}^{n=1} \alpha_i = 1$ is a positive density only if $\alpha_i \ge 0$, $i = 1,\dots,n+1$, then $\conof { q} = \MF{V_p^1}$.
\end{enumerate}
\end{proposition}

\begin{proof}\begin{enumerate}
\item (Cf. \cite{santacroce|siri|trivellato:2015}) 
We use Portmanteu Theorem \ref{prop:maxexp-pormanteau}.\ref{prop:maxexp-pormanteau-6}. Given $q_1,q_2 \in \maxexpat p$, $q_1 = \euler_p(U_1)$ and $q_2 = \euler_p(U_2)$ consider the convex combination $q_\theta = (1-\theta) q_1 + \theta q_2$, $0 < \theta <1$. From the convexity of $x \mapsto x^{1+\epsilon}$ we derive
\begin{align*}
  \int \left(\frac {q_\theta}p\right)^{1+\epsilon}p &= \int \left(\frac {(1-\theta)q_1+\theta q_2}p\right)^{1+\epsilon}p \\
&\le (1-\theta) \int \left(\frac {q_1}p\right)^{1+\epsilon}p + \theta \int \left(\frac {q_2}p\right)^{1+\epsilon}p \ , 
\end{align*}
where both integrals are finite for some $\epsilon > 0$.

From the convexity of $x \mapsto x^{-\epsilon}$ we derive
\begin{align*}
   \int \left(\frac p{q_\theta}\right)^{1+\epsilon}q_\theta &= \int \left(\frac p{(1-\theta)q_1+\theta q_2}\right)^{1+\epsilon}((1-\theta)q_1+\theta q_2) \\ &= \int p^{1+\epsilon}((1-\theta)q_1+\theta q_2)^{-\epsilon} \\ &\le
(1-\theta) \int p^{1+\epsilon}q_1^{-\epsilon} + \theta \int p^{1+\epsilon} q_2^{-\epsilon}
\\ &= (1-\theta) \int \left(\frac p {q_1}\right)^{1+\epsilon}q_1 + \theta \int \left(\frac p {q_2}\right)^{1+\epsilon}q_2 \ ,
\end{align*}
where both integrals are finite for some $\epsilon > 0$. 
\item Consider the vector space $V_p^2$ generated in $\preBspaceat p$ by $\frac {q_i}p - 1$, $i = 1,\dots,n+1$. As $V_p^1 = (V_p^2)^\perp$, we have $(V_p^1)^\perp = V_p^2$ so that $\MF{V_p} = V_p^2 \cap \maxexp$. A generic $v \in V_p^2$ is a linear combination $v = \sum_{i=1}^{n+1} \alpha_j (\frac {q_j}p - 1)$, and $v=\frac \barq p -1$ for a density $\barq$ if $\sum_{i=1}^{n=1} \alpha_j = 1$. In particular this is true for each $\barq \in \conof {q}$.  
\item If the assumption holds true, all $\alpha_i$'s that produce a density are nonnegative.
\end{enumerate}
\qed
\end{proof}

The exponential transport $\etransport p {\barq} U = U - \expectat {\barq} U$, ${\barq} \in \conof {q}$ acts on $V_p^1$ as $U - \sum_{j=1}^{n+1} \expectat {q_j} U = U$, so that 
\begin{multline*}
\etransport p \barq V_p^1 = \setof{\etransport p \barq U}{U \in \Bspaceat p, \expectat {q_i} U = 0, i=1,\dots,n+1} = \\ \setof{V \in \Bspaceat {\barq}}{\expectat {q_i} {V} = 0, i=1,\dots,n+1} = V_q^1  
\end{multline*}

We define the \emph{exponential family orthogonal to $\conof{q}$} to be $\EF{V_{\barq}^1} = \setof{\euler_{\barq}(U)}{U \in V_{\barq}^1}$ for any $\barq \in \conof{q}$. Note that the same exponential family can be expressed at any $p$, in which case the base space is 
\begin{multline*}
  V_p^1 = \etransport {\barq} p V_{\barq}^1 = \setof{\etransport {\barq} p U}{U \in \Bspaceat {\barq}, \expectat {q_i} U = 0, i=1,\dots,n+1} = \\ \setof{V \in \Bspaceat p}{U \in \expectat {q_i} V = \expectat {\barq} V, i=1,\dots,n+1}.
\end{multline*}

The families $\EF{V_p^1}$, $\MF{V_p^1}$ described above form a couple as discussed in Section~\ref{sec:submodels} above.

\section{Finite dimensional approximations by projection}\label{sec:proj}

We now have all the tools we need to derive finite dimensional approximations of infinite dimensional evolution equations for probability measures, such as the ones we have highlighted in Section \ref{sec:ideq} from probability theory, signal processing, social sciences, physics and quantum theory. This can be done with the rigorous infinite dimensional manifold structure from G. Pistone and co-authors we have summarized in the previous sections. 

As we have mentioned in the introduction, this has been done in the past by D. Brigo and co-authors in \cite{brigobernoulli,brigoieee,armstrongbrigomcss} for the filtering problem and in \cite{brigogyor,brigoarhus} for the Fokker-Planck equation, but using the whole $L^2$ space as superstructure, without specifically investigating the geometric structures at play in the infinite-dimensional environment, except for the enveloping exponential manifold discussion in \cite{brigobernoulli}.

Here we will develop the case of the Fokker-Planck PDE since, as we explained in Section \ref{sec:ideq}, this is really the element that brings about infinite dimensionality even in the more complex cases of signal processing and quantum theory stochastic PDEs. The Fokker-Planck equation is thus the ideal benchmark case where one can study dimensionality reduction at the crossroad of different areas. 

We should also mention briefly that the SPDE case we do not treat here involves infinite-dimensional evolution equations  driven by noise and rough paths. The driving rough paths motivate possibly different types of projections related to stochastic differential geometry and introduce different notions of optimality of the projection of the equation solution. We do not have this problem here, since our Fokker-Planck benchmark case will simply be a PDE and will not be driven by noise, but for the general case see the forthcoming paper by J. Armstrong and co-authors \cite{armstrongbrigoicms} in this same volume. 

Before turning to the Fokker-Planck equation, however, we first consider our running example of Section~\ref{sec:running}.

\subsection{Finite dimensional approximation for the heat equation}

With the notations of Definition~\ref{def:orlich-sobolev}, let $p$ be a density in the $W_{\Phi}^1$-exponential family, $p \in \dmaxexpat M$, that is $p = \euler^{U - K_M(U)} \cdot M$ and $U \in \dsdomainat M = \sdomainat M \cap \Bspaceat M \cap W_{\Phi}^1$.

Let $\mathcal A p$ be the non-linear differential operator $p^{-1} \mathcal L^* p$ where $\mathcal L^*$ is the differential operator for our running example equation of Section~\ref{sec:running}, where we assume bounded and uniformly positive definite matrix of coefficients $[a_{ij}]$. Namely, we are considering the anisotropic heath equation.

\begin{equation*}
  \mathcal A p(x) = p(x)^{-1} \sum_{i,j=1}^d \partiald {x_i} \left(a_{ij}(x) \partiald {x_j} p(x)\right), \quad x \in \reals^d \ .
\end{equation*}

Conditions on the coefficients $[a_{ij}]$ are to be given in order to show that the operator on a sufficiently large domain $\mathcal D$ is a section of the differentiable mixture bundle, namely $\mathcal A(p) \in \predBspaceat p$, $p \in \mathcal D \subset \dmaxexpat M$. We do not want to discuss here such conditions. It was done in \cite{lods|pistone:2015} for the special case of the Laplacian, and we assume this property from now on. Note that the zero expectation condition is trivially verified by
\begin{equation*}
  \expectat p {\mathcal A p(x)} = \int \mathcal L^\ast p(x) \ dx = \int p(x) \mathcal L 1 \ dx = 0.
\end{equation*}

Recall that the differentiable predual bundle has an affine atlas of charts, see Definition \ref{def:diffbundle}\eqref{def:diffbundle2}. The chart centered at $p$ is
\begin{equation*}
 \prescript{*}{}\sigma_p \colon \dmixbundleat M \ni (q,V) \mapsto \left(s_p(q), \mtransport q p V\right) \in \dBspaceat p \times \predBspaceat p.
\end{equation*}
where the exponential chart is $s_p(q) = \log\frac q p - \expectat p {\log\frac q p}$ and the linear transport $\mtransport q p \colon \predBspaceat q \to \predBspaceat p$ is defined by $V \mapsto \frac q p V$. 

\begin{example}
In the chart centered at $M$,
\begin{equation*}
 \prescript{*}{}\sigma_M(\euler^{U - K_M(U)} \cdot M,V) = \left(U, \euler^{U - K_M(U)}V\right) \in \dBspaceat M \times \predBspaceat M.
\end{equation*}
 It follows that the expression of the operator $\mathcal A$ in the charts centered at $M$ is of the form 
 \begin{multline*}
   U \mapsto \widehat {\mathcal A}_M(U) = \euler^{U - K_M(U)} \mathcal A(\euler^{U - K_M(U)} \cdot M) = \\  \frac{\euler^{U - K_M(U)}}{\euler^{U - K_M(U)}\cdot M} \mathcal L^* (\euler^{U - K_M(U)} \cdot M) = 
M^{-1} \mathcal L^*(\euler^{U - K_M(U)} \cdot M)
 \end{multline*}

The computation in Equation \eqref{runningcomputation} gives

\begin{multline*}
  M^{-1} \mathcal L^*(\euler^{U - K_M(U)} \cdot M) = \\ \euler^{U - K_M(U)} \sum_{i,j=1}^d \partiald {x_i}\left[a_{ij}(x)\left(\partiald {x_j} U(x) - x_j\right)\right] + \\ \euler^{U - K_M(U)} \sum_{i,j=1}^d a_{ij}(x)\left(\partiald {x_i} U(x) - x_i\right)\left(\partiald {x_j} U(x) - x_j\right).
\end{multline*}
\end{example}

We want now to consider the weak form of the operator, which is defined for each $V \in \dBspaceat p$ by
\begin{align*}
  \scalarat p {\mathcal A p}{V} &= \int p(x)dx \  p(x)^{-1} \sum_{i,j=1}^d \partiald {x_i} a_{ij}(x) \partiald {x_j} p(x) \ V(x) \\ &= \sum_{i,j=1}^d \int dx \  \partiald {x_i} a_{ij}(x) \partiald {x_j} p(x) \ V(x) \\ &= - \sum_{i,j=1}^d \int dx \ a_{ij}(x) \partiald {x_j} p(x) \partiald {x_i} V(x).
\end{align*}
Note that the weak form we have defined at each $p$ is just the usual weak form of the operator $\mathcal L^\ast$, so that it is negative definite. If we proceed with the exponential charts and Equation \eqref{eq:runningfirst} we get
\begin{align*}
  \scalarat p {\mathcal A p}{V} &= - \sum_{i,j=1}^d \int p(x)dx \ a_{ij}(x) (\partiald {x_j} U(x) - x_j) \partiald {x_i} V(x) \\ &= \sum_{i,j=1}^d \scalarat p {a_{ij}(X)(X_j - \partial_j U)} {\partial_i V} \\ &= \sum_{i,j=1}^d \scalarat p {a_{ij}(X)X_j} {\partial_i V}  - \sum_{i,j=1}^d \scalarat p {a_{ij}(X)\partial_j U } {\partial_i V}.
\end{align*}

Note that $U$ belongs to $\dBspaceat M$, so that $X_j$ and $\partial_j U$ both belong to $\Lexp M$. It is sufficient to assume $[a_{ij}]$ uniformly bounded. Weaker conditions are allowed, as we actually need to assume that the multiplication operator $W \mapsto a_{ij}(X)W$ maps $\LlogL p$ into itself for all $p$.

To define a Galerkin-style projection, we want finite dimensional subspaces $V_n(p)$ of the fibers $\dBspaceat p$. Such subspaces are obtained from a reference one $V_n(M)$ via the application of the exponential parallel transport. Assume $V_n \in \dBspaceat M$ is a vector space of dimension $n$ and take $U \in V_n$ and $V \in V_n(p) = \etransport M p V_n$. As the exponential transport has no effect on the partial derivatives, we have for $U,V \in \dBspaceat M$

\begin{align*}
  \scalarat p {\mathcal A p}{\etransport M p V} &= \sum_{i,j=1}^d \scalarat p {a_{ij}(X)(X_j - \partial_j)U} {\partial_i V} \\ 
&= \sum_{i,j=1}^d \scalarat p {a_{ij}(X) X_j} {\partial_i V} - \sum_{i,j=1}^d \scalarat p {a_{ij}(x)\partial_j U} {\partial_i V} 
\end{align*}

Let $(W_1,\dots,W_n)$ be a basis of $V_n$, so that $(W_1 - \expectat p {W_1},\dots,W_n - \expectat p {W_n})$ is a basis of $V_n(p)$. We can write

\begin{align*}
  U &= \sum_{h=1}^n \theta_h W_h \\
  V &= \sum_{k=1}^n \alpha_k W_k  \\
\end{align*}
and
\begin{equation*}
  \scalarat p {\mathcal A p}{\etransport M p V} = \sum_{h,k=1}^n \theta_h \alpha_k \sum_{i,j=1}^d \scalarat p {a_{ij}(X)(X_j - \partial_j) W_h} {\partial_i W_k}
\end{equation*}
Equivalently,
\begin{equation*}
 \scalarat p {\mathcal A p}{\etransport M p W_k} = \sum_{h=1}^n \theta_h \sum_{i,j=1}^d \scalarat p {a_{ij}(X)(X_j - \partial_j) W_h} {\partial_i W_k}, \quad k = 1,\dots, n
\end{equation*}

In the exponential family of densities of the form
\begin{equation*}
  p = \expof{\sum_{h=1}^n \theta_k W_k - \psi(\theta)} \cdot M
\end{equation*}
we look for a curve $t \mapsto p(t)$ whose score $Dp(t)$ is such that
\begin{equation}\label{eq:galerkin1}
  \scalarat {p(t)} {Dp(t) - \mathcal Ap(t)} {\etransport M {p(t)} W_k} = 0, \quad k=1,\dots,n.
\end{equation}
In fact, the curve $t \mapsto (p(t),Dp(t) - \mathcal Ap(t))$ belongs to a statistical bundle, hence has to be checked against a moving frame. The score can be written in the moving frame as
\begin{equation*}
  Dp(t) = \frac {\dot p(t)}{p(t)} = \sum_{h=1}^n \dot \theta_h(t) \etransport M {p(t)} W_h
\end{equation*}
so that 
\begin{equation*}
  \scalarat {p(t)} {Dp(t)} {\etransport M p W_k} =  \sum_{h=1}^n \dot \theta_h(t)  \scalarat {p_{\theta(t)}} {\etransport M {p(t)} W_h} {\etransport M p W_k} =  \sum_{h=1}^n g_{hk}(t)\dot \theta_h(t),
\end{equation*}
where we have used the Fisher matrix
\begin{equation*}
  g(\theta) = [\scalarat {p_{\theta}} {\etransport M {p_\theta} W_h} {\etransport M {p_\theta} W_k}]_{h,k} = [\covat {p_{\theta}} {W_h} {W_k}]_{h,k} = \hessianof{\psi(\theta)}.
\end{equation*}
Equation \eqref{eq:galerkin1} becomes
\begin{equation}\label{eq:galerkin2}
  \sum_{h=1}^n g_{kh}(\theta(t)) \dot\theta_h(t) = \sum_{h=1}^n \sum_{i,j=1}^d \scalarat {p_{\theta(t)}} {a_{ij}(X)(X_j - \partial_j) W_h} {\partial_i W_k} \theta_h(t),
\end{equation}
for all $k=1,\dots,n$.

If the inverse Fisher matrix is $g(\theta)^{-1} = [g^{lk}(\theta)]$, we can multiply the equation by $g^{lk}(\theta(t))$ and sum over $k$ to get the system of non linear differential differential equations:
\begin{equation}\label{eq:galerkinheat}
\dot\theta_l(t) = \sum_{h=1}^n \sum_{i,j=1}^d \scalarat {p_{\theta(t)}} {a_{ij}(X)(X_j - \partial_j) W_h} {\partial_i \sum_{k=1}^d g^{lk}(\theta(t))W_k} \theta_h(t),
\end{equation}
for all $l=1,\dots,n$.

We have shown that it is possible, at least in principle, to derive Galerkin-type approximations of our running example. To proceed to a practical implementation it would be necessary to choose a suitable basis $(W_1,\dots,W_n)$ for which the Galekin equation \eqref{eq:galerkinheat} is computable. 

We now turn to examine from a different perspective a second example, the Fokker-Plank equation. 

\subsection{Fokker-Planck Equation in statistical manifold coordinates}

%We treat the scalar Fokker-Planck equation to contain notation, namely we assume that $x \in \Rx^1$, $N=1$. 

We could apply the same techniques we used in the running example pari passu to the Fokker-Planck equation
\eqref{eq:FP}, keeping in mind the definition of the related operators ${\cal L}$ and ${\cal L}^\ast$. 
However, we will proceed at a low pace given the more complicated nature of \eqref{eq:FP} compared to our running example. We proceed step by step by showing how the specific structure of \eqref{eq:FP} is dealt with in the statistical manifold context of this paper.

We may want to avoid using necessarily the Gaussian density $M$ as background density, so for simplicity in this section  we work in a single chart and assume the equation is written until the first exit time from the manifold. For example, again in the case $c_1(x)=x, c_2(x)=x^2, \ldots,c_n(x) =x^n$, $n$ even natural number, this would correspond to the first exit time from \{$\theta_n <0\}$. We might avoid the exit time by introducing a suitable background density, for example $M_{1,n+2}$, but for simplicity we do not assume a background density in the derivation. We will discuss again the possible use of a background density when considering the ${\cal L}$ eigenfunctions later.  

Now we rewrite equation (\ref{eq:FP}) in exponential
coordinates.
%
%We assume that the curve $t \mapsto p_t$ in ${\cal M}$
%is differentiable. In such a case we compute from equation
%(\ref{eq:FP}) the representation of the tangent vectors
%in exponential coordinates.
Consider as local reference density the
solution $p_t$ of FPE at time $t$. We are now working around $p_t$.
Consider a curve around $p_t$ corresponding to the
solution of FPE around time $t$ expressed in $B_{p_t}$
coordinates:
\begin{eqnarray*}
h &\mapsto& s_{p_t}(p_{t+h})=: u_h.
\end{eqnarray*}
The function $u_h$ represents the expression in coordinates of
the density
\begin{eqnarray} \label{coordin}
p_{t+h} = \exp[u_h - K_{p_t}(u_h)] p_t =: e_h p_t.
\end{eqnarray}
Now consider FPE around $t$, i.e.
\begin{eqnarray*}
\frac{\partial p_{t+h}}{\partial h} = {\cal L}_{t+h}^\ast p_{t+h}.
\end{eqnarray*}
Substitute (\ref{coordin}) in this last equation in order to obtain
\begin{eqnarray*}
\frac{\partial e_h p_t}{\partial h} = {\cal L}_{t+h}^\ast (e_h p_t).
\end{eqnarray*}
Write
\begin{eqnarray*}
\frac{\partial e_h}{\partial h} = \frac{{\cal L}_{t+h}^\ast (e_h
p_t)}{p_t}
\end{eqnarray*}
and set $h=0$, since we are concerned with the behavior in $t$.
Notice that $e_0=\exp[u_0-K_{p_t}(u_0)]=\exp(0)=1$,
and that
\begin{eqnarray*}
\left.\frac{\partial e_h}{\partial h}\right|_{h=0} = \{e_h
\frac{\partial [u_h - K_{p_t}(u_h)]}{\partial h}\}|_{h=0}
=\frac{\partial [u_h - K_{p_t}(u_h)]}{\partial h}|_{h=0}.
\end{eqnarray*}
Moreover, by straightforward computations (write explicitly
the map $K_{p_t}$, use
$u_h=s_{p_t}(p_{t+h})$ and differentiate wrt $h$ under the expectation
$E_{p_t}$) one verifies
\begin{eqnarray*}
\left. \frac{\partial K_{p_t}(u_h)}{\partial h} \right |_{h=0}=0,
\end{eqnarray*}
so that
\begin{equation} \label{equation-u}
\left. \frac{\partial u_h}{\partial h} \right |_{h=0}
= \frac{{\cal L}_t^\ast p_t}{p_t}
\end{equation}
is the formal representation in exponential coordinates
of the vector in the statistical exponential (vector) bundle $\expbundle$ at $p_t$.
Notice that, again by straightforward computations, and omitting the time arguments in $f$ and $a$ for brevity,
\begin{eqnarray} \label {alpha-def}
   \alpha_t := \alpha_t(p) = \frac{{\cal L}_t^\ast p}
          {p}
   &=&  - \sum_{i=1}^N \left( f_i\, \frac{\partial}{\partial x_i}(\log p)
    + \frac{\partial f_i}{\partial x_i}\right) +
   \\ \nonumber \\ \nonumber
   &+& \half \sum_{i,j=1}^N \left[\,
   a_{ij}\,
   \frac{\partial^2}{\partial x_i \partial x_j}(\log p)
   + a_{ij}\, \frac{\partial}{\partial x_i}(\log p) \frac{\partial}{\partial x_j}(\log p)\,+ \right.
   \\ \nonumber \\  \nonumber
   &+&  \left . 2\, \frac{\partial a_{ij}}{\partial x_j}\,
   \frac{\partial}{\partial x_i}(\log p)
   + \frac{\partial^2 a_{ij}}{\partial x_i \partial x_j} \,\right]\ .
\end{eqnarray}
Summarizing: consider the curve expressing FPE around $p_t$ in
$B_{p_t}$ coordinates. Its tangent vector/fiber in the statistical exponential bundle $\expbundle$ at $p_t$  is given
by $\alpha_t$.
Under suitable assumptions on the coefficients $f_t$ and $a_t$
the function
$\alpha_t$ belongs to $B_{p_t}$, according to the
convention that locally identifies the tangent bundle of a normed space
with the normed space itself.
To render the computation not only formal
we need $\alpha_t$ to be really a tangent vector/fiber for our
bundle structure. This in turn requires the curve
$t \mapsto p_t$ to be differentiable in the proper sense.
Below we give a regularity result
expressing a condition under which this happens and whose proof
is immediate.
Moreover, we give a condition which can be used to check
whether the evolution stays in a given submanifold.
\begin{proposition}
[Regularity and finite dimensionality
of the solution of FPE] \label{th:reg}
\hspace{2cm}
\begin{itemize}
\item[(i)] If the map $t \mapsto p_t$ is differentiable
in the manifold ${\cal E}$ then $\alpha_t$ given in
eq. (\ref{alpha-def}) is a tangent vector.
\item[(ii)] If the map
$t \mapsto \alpha_t$ is continuous
at $t_0$ into $L^\Phi$,
then $t \mapsto p_t$ is differentiable at $t_0$
as a map into ${\cal E}$.
\item[(iii)]
Let be given a submanifold ${\cal N}$ such
that $p_0 \in {\cal N}$. If the previous condition
is satisfied and
\begin{displaymath}
\frac{{\cal L}_t^\ast p}{p}
\end{displaymath}
is tangent to ${\cal N}$ at $p$ for all $p \in {\cal N}$,
then $p_t$ evolves in ${\cal N}$.
\end{itemize}
\end{proposition}

Sufficient conditions under which condition (ii) in the
proposition holds are related to boundedness for all possible $T>0$ and $i,j$ of $f$, $\partial_{x_i} f$,
$a$, $\partial_{x_i} a$, $\partial^2_{x_i x_j} a$ in $[0 \ \ T] \times \Rx$
plus classical assumptions  ensuring  (D).
This follows from the fact that if $\alpha_t(x)$ is continuous
and bounded in both $t$ and $x$, then it is continuous as
a map $t \mapsto \alpha_t$ from $[0 \ \ T]$ to $L^\Phi$.

\subsection{Projection of the infinite dimensional Fokker-Planck equation}
\label{projfm}
The references \cite{brigogyor} and \cite{2009arXiv0901.1308B} present a few examples of SDEs whose densities satisfy Proposition \ref{th:reg}. These are special cases of SDEs whose solution density, satisfying the related Fokker Planck equation, stays in a finite dimensional exponential family.  Examples include the trivial linear Gaussian SDEs case, nonlinear SDEs with solutions having unit variance Gaussian law,  and SDEs with prescribed diffusion coefficient $\sigma_t(x)$ and with prescribed stationary density in a given exponential family, among others.

However, in general the evolution of the density of the solution of a given SDE does not happen to satisfy Proposition \ref{th:reg} and one has to deal with the infinite dimensionality by choosing a finite dimensional approximation of the solution of the Fokker Planck equation.  We will now derive such an approximation based on a projection argument.

In reaching equation (\ref{equation-u}) we assumed
implicitly a few facts.
We are assuming that there always exists a neighborhood of
$h=0$ such that in this neighborhood $p_{t+h} \in {\cal E}({p_t})$.
Conditions under which this happens will be examined in the future.
We only remark that when projecting on a finite dimensional
exponential manifold, these conditions are not necessary for the
projected equation to exist and make sense, see below.
Neither we need equation (\ref{equation-u}) to have a solution
to obtain existence of the solutions of the projected equation.
Now we shall project this equation on a finite dimensional
parametrized exponential manifold $\EF c$.
We will assume the following on the family  $\EF c$ (see \cite{brigobernoulli} 
for other more specific assumptions):

\[ (E) \hspace{1cm} \mbox{We assume} \ \ \ \   c \in C^2.\]

A rapid projection computation based on Formula \eqref{eq:projexpu} and involving integration by parts between ${\cal L}$ and ${\cal L}^\ast$ and standard results
on the normalization constant $\psi(\theta)$ of exponential families
(such as $\partial_{\theta_i} \psi(\theta) = E_{\theta} c_i$)
yields
\begin{eqnarray*}
 {\cal P}_{t,\theta}
  &:=& \Pi_\theta
  \left[\frac{{\cal L}_t^\ast p(\cdot,\theta)}
          {p(\cdot,\theta)}\right] =
    E_{\theta}[{\cal L}_t c]^T \  g^{-1}(\theta)
        \ [c(\cdot) - E_{\theta} c],
\end{eqnarray*}
where integrals of vector functions are meant to be applied to their
components.
Note that this map is regular in $\theta$ under reasonable assumptions
on $f,a$ and $c$.
At this point we project equation (\ref{equation-u})
via this projection. By remembering expression (\ref{eq:tangspexp})
for tangent vectors and the above formula for the projection
we obtain the  following ($n$--dimensional)
ordinary differential equation (in vector form)
in the coordinates of the manifold $\EF c$:
\begin{eqnarray} \label{PFPEPAR}
   \dot{\theta}_t
   = g^{-1}(\theta_t)\; E_{\theta_t}\{{\cal L}_t\; c\}.
\end{eqnarray}
Notice that, as anticipated above,
equation (\ref{PFPEPAR})
is well defined and admits locally a unique solution
if  the following condition (ensuring existence of the
norm of $\alpha_t(p(\cdot,\theta_t))$ associated to the inner product
$\covat {p_{\theta_t}} {\cdot}{\cdot}$)
holds:
\begin{eqnarray} \label {alpha-cond}
(F) \hspace{1cm}   &&E_{\theta}\{\alpha_{t,\theta}^2\}<\infty \;\;
                  \forall \theta \in \Theta, \ \forall t \ge 0, 
   \\ \nonumber \\ \nonumber
   &&\alpha_{t,\theta} := \frac{{\cal L}_t^\ast p(\cdot,\theta)}
          {p(\cdot,\theta)} =
     - \sum_{i=1}^N  \left( f_i\,\frac{\partial}{\partial x_i}(\theta^T c)
    + \frac{\partial f_i}{\partial x_i}\right) +
       \\ \nonumber \\   \nonumber
   &&\hspace{1cm}  + \half  \sum_{i,j=1}^N \left[\,
   a_{ij}\,
   \frac{\partial^2}{\partial x_i \partial x_j}(\theta^T c)
   + a_{ij}\, \frac{\partial}{\partial x_i}(\theta^T c) \frac{\partial}{\partial x_j}(\theta^T c)\,+ \right.
   \\ \nonumber \\  \nonumber
   && \hspace{1cm} + \left. 2\, \frac{\partial a_{ij}}{\partial x_j}\,
   \frac{\partial}{\partial x_i}(\theta^T c)
   + \frac{\partial^2 a_{ij}}{\partial x_i \partial x_j} \,\right]\ .
\end{eqnarray}
We will assume such condition to hold in the following.
 Sufficient explicit conditions for (F) to hold for $\EF c$ can be easily given. For example, (F)  holds if $f$ and its first derivatives with respect to $x$, $a$ and its first and second derivatives with respect to $x$,  and $c$ and its first and second derivatives have at most polynomial growth, and if densities in $\EF c$ integrate any polynomial, see for example \cite{brigobernoulli}.

We have thus proven the following
\begin{proposition}
[Projected evolution of the density of an It\^o diffusion]
\label{proj-law}
Assume assumptions (A), (B),(C), (E) and (F) on the coefficients
$f, a$, on the initial condition $X_0$ of the It\^o diffusion $X$,
and on the sufficient statistics
$c_1,\dots,c_n$ of the exponential family $\EF c$ are satisfied.
Then the projection of Fokker-Planck equation describing the
evolution of $p_t = p_{X_t}$ onto $\EF c$  reads, in
$B_{p_t}$ coordinates:
\begin{eqnarray}\label{PFPE}
[c(\cdot) - E_{\theta_t}c]^T \dot{\theta}_t =
E_{\theta_t}[{\cal L}c]^T \  g^{-1}(\theta_t)
\ [c(\cdot) - E_{\theta_t}c],
\end{eqnarray}
and the differential equation describing the evolution of the
parameters for the projected density--evolution is
\begin{eqnarray*}
   \dot{\theta}_t
   = g^{-1}(\theta_t)\; E_{\theta_t}\{{\cal L}_t\; c\}.
\end{eqnarray*}
\end{proposition}
Notice that the projected equations exist under conditions
which are more general than conditions for existence of the
solution of the original Fokker-Planck equation.
For more details see \cite{brigogyor}.
Notice also that this equation is substantially the same we had derived in the running example with a Galerkin-inspired approach: Compare \eqref{PFPEPAR} with \eqref{eq:galerkinheat} after viewing the right hand side of \eqref{eq:galerkinheat} as coming from an integration by parts. 

Finally, we point out a result previously given in 
\cite{brigogyor} and \cite{2009arXiv0901.1308B},  see also \cite{brigospl}, where it is explained, for the case $N=1$, how one can build a SDE whose solution has a density evolving exactly as the projected density $t \mapsto p(\cdot,\theta_t)$. This allows one to design SDEs whose marginal laws evolve in a given exponential family.  Here we only briefly state the related result:

\begin{proposition} [Interpretation of the projected density--evolution as the exact density of a different SDE]
Assume assumptions (A), (B), (C), (E) and (F) on the coefficients
$f, a = \sigma^2$ and on the initial condition $X_0$
of the It\^o diffusion 
\[ dX_t = f_t(X_t) dt + \sigma_t(X_t) dW_t, \ \ X_0\]
and on the sufficient statistics
$c$ of the exponential family  $\EF c$ are satisfied.
Let $p(\cdot,\theta_t)$ be the projected density evolution, according to
proposition \ref{proj-law}. Define
\begin{eqnarray*}
d Y_t &=& u^\ast_t(Y_t) dt + \sigma_t(Y_t) dW_t, \\ \\
 u^\ast_t(x) &:=& \half \frac{\partial a_t}{\partial x}(x) +
          \half a_t(x) \theta_t^T  \frac{\partial c}{\partial x}(x)
          + \\ \nonumber \\ \nonumber
&& - E_{\theta_t}\{{\cal L}_t c\}^T g^{-1}(\theta_t)
                 \int_{-\infty}^x (c(y) - E_{\theta_t}c)
                 \ \exp[\theta_t^T (c(y) - c(x))] dy.
\end{eqnarray*}
Then $Y$ is an It\^o diffusion whose density--evolution coincides with the
projected density--evolution
$p(\cdot,\theta_t)$ of $X_t$ onto  $\EF c$.
\end{proposition}

\subsection{Quality of the finite dimensional approximation}
In order to assess how good the projection is locally, and to have a measure for how far the projected evolution is, locally, from the original one, we now define a local projection residual as the duality-based norm of the Fokker Planck infinite dimensional vector field minus its finite-dimensional orthogonal projection. Define the vector field minus its projection as
\[ \varepsilon_t(\theta) := \frac{{\cal L}_t^\ast p(\cdot,\theta)}
          {p(\cdot,\theta)} - \Pi_\theta
  \left[\frac{{\cal L}_t^\ast p(\cdot,\theta)}
          {p(\cdot,\theta)}\right] .\]
          Then the projection residual $R_t$ is defined as 
\[ R^2_t := \covat {p_{\theta}}{\varepsilon_t(\theta)}{\varepsilon_t(\theta)} = \left\langle \varepsilon_t(\theta), \varepsilon_t(\theta)\right\rangle_{p(\cdot,\theta)}
\]
and can be computed jointly with the projected equation evolution \eqref{PFPEPAR} to have a local measure of the goodness of the approximation involved in the projection.

Monitoring the projection residual and its peaks can be helpful in tracking the projection method performance, see also  \cite{brigobernoulli,brigoieee} for examples of $L^2$-based projection residuals in the more complex case of the Kushner-Stratonovich equations of nonlinear filtering. However, the projection residual only allows for a local approximation error numerical analysis. To have an idea of how good the approximation is we need to relate it to the global approximation error.

We could define the global approximation error as follows. Rather than projecting the Fokker Planck equation vector field instant by instant, we could project the true solution as a point onto the exponential family $\EF c$. To appreciate the difference with what we have done so far, let us recap the method we have followed so far, which we call ``vector field projection''. We denote time steps with $0,1,2,\ldots$ for simplicity but in the real equation they correspond to infinitesimal time steps. To make the point, we are artificially separating projection and propagation and the local and global errors. This is not completely precise but allows us to make an important point on our method. 

\begin{itemize}
\item Assume at time 0 we have $p_0(x) = p(x;\theta_0)$, so we start from the family.
\item Now the vector field of Fokker Planck $\frac{{\cal L}^\ast p(\cdot,\theta_0)}{p(\cdot,\theta_0)}$ is not in the tangent space of $\EF c$ in general and therefore would bring us out of the exponential family at time 1. To stay in the exponential family, we project this vector field onto the tangent space of $\EF c$ and follow the projected vector for the evolution, moving on the tangent space to time 1. By doing this, we get a new $p(\cdot,\theta_1)$ on the manifold. 
\item Now we start again. We apply the vector field of the Fokker Planck equation to $p(\cdot,\theta_1)$. Note that this is not right if comparing with the true evolution. We are applying the vector field to the wrong point at time 1, because $p(\cdot,\theta_1)$ is not  the true $p_1$, and now we are not applying the vector field to $p_1$ but to $p(\cdot,\theta_1)$. But even starting from $p(\cdot,\theta_1)$, the vector field
$\frac{{\cal L}^\ast p(\cdot,\theta_1)}{p(\cdot,\theta_1)}$ is not in the tangent space of $\EF c$ in general and therefore would bring us out of EF. To stay in EF, we project this vector field onto the tangent space of the exponential family and follow the projected vector for the evolution, moving on the tangent space. By doing this, we get a new $p(\cdot,\theta_2)$ at time 2 on the manifold. 
\item We continue like this and obtain an evolution of the manifold, but none of the projections was based on projecting the vector field starting from the true solution, except for the first step. 
\end{itemize}

This method has two types of approximations, so to speak: on one hand, we approximate the true equation vector field with a projection. On the other hand, we apply the true equation vector field not to the true solution but already to an approximated solution coming from the previous steps. The two steps are related in the limit, clearly, and with some very sophisticated analysis one might be able to bound the global error based on the local one.  However, let us continue with the artificial setting with separate steps. We can say that while it is possible to measure locally the error in the first type of approximation, for example via $R_t$ above, it is difficult to measure the effect  of the second one, unless one obtains a very precise approximation of the true solutions by some other method and then compares the outputs. But if one has the true solution to a very good precision already, there is clearly no point in finding a finite dimensional approximation. 

If we leave the global approximation error analysis aside for a minute, the big advantage of the above method is that it does not require us to know the true solution of the Fokker Planck equation to be implemented. Indeed, Equation \eqref{PFPEPAR} works perfectly well without knowing the true solution $p_t$. 

As we mentioned above, to study the global error, we now introduce a second projection  method. This one will require us to know the true solution, so as an approximation method it will be pointless. However, it will help us with the global error analysis, and a modification of the method based on the assumed density approximation will allow us to find an algorithm that does not require the true solution.

This method works as follows.

\begin{itemize}
\item Assume at time 0 we have $p_0(x) = p(x;\theta_0)$, so we start from the family.
\item Now the vector field of Fokker Planck $\frac{{\cal L}^\ast p(\cdot,\theta_0)}{p(\cdot,\theta_0)}$ is not in the tangent space of $\EF c$ in general and therefore would bring us out of the exponential family at time 1. We accept this, follow it, and move to $p_1$ outside $EF(c)$. To go back to EF, we project $p_1$ onto the exponential family by minimizing the divergence, or Kullback Leibler information of $p_1$ with respect to $\EF c$, finding the orthogonal projection of $p_1$ on EF. It is well known that the orthogonal projection in Kullback Leibler divergence is obtained by matching the sufficient statistics expectations of the true density. Namely, the projection is the particular exponential density of $\EF c$ with $c$-expectations 
          \[ \eta_1 = E_{p_1}[c] . \]
          See for example \cite{brigoime} for a quick proof and an application to filtering in discrete time. We know that $\EF c$, besides $\theta$, admits another important coordinate system, the expectation parameters $\eta$. If one defines 
          \[ \eta(\theta) = E_{p(\theta)} [ c]\]
          then $d \eta(\theta) = g(\theta) d\theta$ where $g$ is the Fisher metric.
          Thus, we can take the $\eta_1$ above coming from the true density $p_1$ and look for the exponential density $p(\cdot;\eta_1)$ sharing these $c$-expectations. This will be the closest in Kullback Leibler to $p_1$ in $\EF c$. 
         \item Now from $p_1$ we keep following the true vector field of the Fokker Planck equation, and in general we start from outside the manifold $\EF c$  and we stay outside. 
   We reach $p_2$. Now again we project $p_2$ onto the exponential family in Kullback Leibler, finding $\eta_2 = E_{p_2}[c]$ and the projection is the exponential density $p(\cdot;\eta_2)$. 
       \item We continue like this
       \end{itemize}

 The advantage of this method compared to the previous vector field based one is that we find at every time the best possible approximation (``maximum likelihood'') of the true solution in EF. The disadvantage is that in order to compute the projection at every time, such as for example $\eta_1 = E_{p_1}[c]$, we need to know the true solution $p_1$ at that time. Clearly if we know the true solution there is no point in developing an approximation by projection in the first place.
 
 However it turns out that we can somewhat combine the two ideas and analyze the error if we invoke the assumed density approximation. This works as follows.
 
\subsection{Maximum likelihood estimation and ${\cal L}$ eigenfunctions}
 
 Consider the second type of projection, namely
 \[ \eta_t = E_{p_t}[c].\]
 Differentiate both sides ($d_t$ here denotes differentiation with respect to time) to obtain
 \[ d_t \eta_t = d_t \int c(x) p_t(x) dx = \int c(x) d_t p_t(x) dx = \int c(x) {\cal L}^\ast_t p_t(x) dx = E_{p_t} [ {\cal L} c] dt\] 
          so that
           \[ d_t \eta_t = E_{p_t} [ {\cal L} c] dt .\]
This last equation is not a closed equation, since $p_t$ in the right hand side is not characterized by $\eta$. Thus, to be solved this equation should be coupled with the original Fokker Planck for $p_t$. Again, this makes this equation useless as an approximation. However, at this point we can close the equation by invoking the assumed density approximation (see \cite{brigobernoulli}): we replace $p_t$ with the exponential density $p(\cdot,\eta_t)$. We obtain
           \[ d_t \tilde{\eta}_t = E_{p(\cdot,\tilde{\eta}_t)} [ {\cal L} c] dt .\]          
 This is now a finite dimensional ODE for the expectation parameters. There is more:
 if we use $d \eta = g(\theta) d\theta$ and substitute, in the $\theta$ coordinates this last equation is the same as our earlier vector field based projected equation
 \eqref{PFPEPAR}. 
 \begin{theorem} Closing the evolution equation for the Kullback Leibler projection of the Fokker Planck solution onto $\EF c$ by forcing an exponential density on the right hand side is equivalent to the approximation based on the vector field projection in Fisher metric. 
 \end{theorem}

We can now attempt an analysis of the error between the best possible projection $\eta_t$ and the vector field based (or equivalently assumed density approximation based) projection $\tilde{\eta}$. To do this, write
\[ \epsilon_t := \eta_t - \tilde{\eta}_t ,\]
expressing the difference between the best possible approximation and the vector field projection / assumed density one, in expectation coordinates. 
Differentiating we see easily that
\[ d \epsilon_t =  (E_{p_t} [ {\cal L} c] -  E_{p(\tilde{\eta}_t)} [ {\cal L} c] )dt .\]
Now suppose that the $c$ statistics in $\EF c$ are chosen among the eigenfunctions of the operator ${\cal L}$, so that
\[ {\cal L} c =  - \Lambda c\]
where $\Lambda$ is a $n \times n$ diagonal matrix with the eigenvalues corresponding to the chosen eigenfunctions. 
Substituing, we obtain
\[ d \epsilon_t =  - \Lambda (E_{p_t} [  c] -  E_{p(\tilde{\eta}_t)} [  c] )dt \]
or 
\[ d \epsilon_t =  - \Lambda  \epsilon_t dt \]
from which 
\[ \epsilon_t = \exp(- \Lambda t) \epsilon_0  \]
so that if we start from the manifold the error is zero, meaning that the vector field projection gives us the best possible approximation. If we don't start from the manifold, ie if $p_0$ is outside $\EF c$, then the difference between the vector field approach and the best possible approximation dies out exponentially fast in time provided we have negative eigenvalues for the chosen eigenfunctions. 

\begin{theorem}[Maximum Likelihood Estimator for the Fokker Planck Equation and Fisher-Rao projection]
The vector field projection approach leading to \eqref{PFPEPAR} provides the best possible approximation of the Fokker Planck equation solution in Kullback Leibler in the family $\EF c$, provided that the sufficient statistics $c$ are chosen among the eigenfunctions of the adjoint operator $\cal L$ of the original Fokker Planck equation, and provided that $\EF c$ is an exponential family when using such eigenfunctions. In other words, under such conditions the Fisher Rao projected equation \eqref{PFPEPAR} provides the exact maximum likelihood estimator for the solution of the Fokker Planck equation in the related exponential family. 
\end{theorem}

The choice or availability of suitable eigenfunctions is not always straightforward, except in a few simple cases. See \cite{pavliotis} for a discussion on eigenfunctions for the Fokker Planck equation. For example, in the one dimensional case $N=1$ where the diffusion is on a bounded domain $[\ell, r]$ with reflecting boundaries and strictly positive diffusion coefficient $\sigma$ then the spectrum of the operator ${\cal L}$ is discrete, there is a stationary density and eigenfunctions can be expressed with respect to this stationary density. In our framework it would be natural to use the stationary density as background density replacing $M(x)$ and then use the eigenfunctions and the related negative real eigenvalues to study the approximation of the Fokker Planck equation.  

For the case $N>1$ only special types of SDEs allow for a specific eigenfunctions/eigenvalue analysis, see for example the Ornstein Uhlenbeck case and SDEs with constant diffusion matrices and drifts associated to potentials in \cite{pavliotis}. Further research is needed to explore the eigenfunctions approach in connection with maximum likelihood. 

\subsection{The direct $L^2$ metric projection}\label{sec:directl2vshellinger}
As we mentioned at the end of Section \ref{sec:expstatL2}, the $L^2$ structure based on square roots of densities (Hellinger distance) and the exponential statistical manifold lead to the same finite dimensional metric on any finite dimensional manifold $p_\theta$ (not just $\EF c$), but the direct $L^2$ metric based on densities rather than their square roots leads to a different finite dimensional metric. Under a background measure $\mu$, by generalizing straightforwardly  \eqref{eq:projexpu} and the related derivation to a general family $p_\theta$ we see that the statistical manifold induces on finite dimensional families the inner product
\[ \covat {p_\theta} {\frac{\partial \log p_\theta}{\partial \theta_i}}{\frac{\partial \log p_\theta}{\partial \theta_j}} = \scalarat {p_\theta} {\frac{\partial \log p_\theta}{\partial \theta_i}}{\frac{\partial \log p_\theta}{\partial \theta_j}} = g_{i,j}(\theta)\]
and the $L^2(\mu)$ based Hellinger distance leads to
\[ {\scalarat \mu {\frac{\partial \sqrt{ p_\theta}}{\partial \theta_i}}{\frac{\partial \sqrt{ p_\theta}}{\partial \theta_j}} = \frac14 g_{i,j}(\theta)}, \]
essentially giving the same Fisher-Rao metric on the finite dimensional manifold. However, the direct metric yields
\[ \scalarat {\mu}  {\frac{\partial { p_\theta}}{\partial \theta_i}} {\frac{\partial { p_\theta}}{\partial \theta_j}} = \gamma_{i,j}(\theta) \neq g_{i,j}(\theta) . \]

This means that the direct metric leads to a different finite dimensional metric $\gamma$, different from the Fisher Rao $g$ given by the Hellinger distance or the statistical manifold structure. This finite dimensional geometry related to $\gamma$ works quite well when projecting infinite dimensional evolution equations on subspaces $\MG q$ generated by mixtures of a given finite set of densities $q$, see \cite{brigol2,armstrongbrigomcss}, and coincides with traditional Galerkin methods based on $L^2$ bases for $p$ directly. The $g$ metric works well when projecting on finite dimensional exponential families such as $\EF c$. The direct metric approach to dimensionality reduction with $\MG q$ mixtures will not be pursued further here given that its induced finite dimensional geometry is different from the statistical manifold induced geometry.

\section{Conclusions and further work}\label{sec:conc}

We have proposed a dimensionality reduction method for infinite--dimensional measure--valued evolution equations such as the Fokker Planck equation or the Kushner-Stratonovich / Duncan Mortensen Zakai equations, with potential applications to signal processing, quantitative finance, heat flows and quantum theory. This dimensionality reduction method is based on a projection coming from a duality argument and allows one to design a finite dimensional approximation for the evolution equation that is optimal locally according to the statistical manifold structure by G. Pistone and co-authors. Clearly the choice of the finite dimensional manifold on which one should project the infinite dimensional equation is crucial, and we proposed finite dimensional exponential and mixture families as in previous works by D. Brigo and co-authors inspired by the $L^2$ structure instead. 

Given the work of N. Newton \cite{MR2948226,MR3126105,MR3356252} on finding an infinite dimensional manifold structure on the space of measures that combines the exponential manifold structure of G. Pistone and co-authors and the $L^2$ full-space structure used by D. Brigo and co-authors, further work is to be done to see how dimensionality reduction based on Newton's framework would look like and would relate to this paper.

It would also be important to see how convergence works when the finite dimensional manifold dimension tends to infinity. Indeed, 
one further natural question is whether it is possible to prove that the finite dimensional approximated solution converges to the infinite dimensional solution when the dimension of the finite dimensional manifold tends to infinity. More precisely, suppose we are given a sequence of functions
 $(c_j)_{j \in \mathbb{N}}$.
Call $c^m := \{c_1 \ c_2 \ \ldots c_m\}$, and assume that
for an infinite subset $\mathbb{M} \subset \mathbb{N}$ and for $m\in \mathbb{M}$ the family $\EF{c^m}$ is a finite dimensional exponential manifold satisfying assumptions (E) and (F). For example, in the monomial case where $c_i(x) = x^i$, we could have that $\mathbb{M}$ is the set of natural even numbers. 
Call $p(\cdot,\theta^m_t)$ the density coming from projection
of Fokker--Planck equation onto $\EF{c^m}$, $m\in \mathbb{M}$.
It is conceivable that in case the infinite sequence $c_k, k \in \mathbb{N}$ is chosen carefully,   one can prove that if $\mathbb{M} \ni m \rightarrow + \infty$ then 
$p(\cdot,\theta^{m}(t)) \rightarrow p_t$ where $p_t$ is the original infinite dimensional density coming from the Fokker Planck equation being approximated. The way to approach this would be to treat the $c_k$ as a basis of an infinite dimensional space and to use Sobolev spaces and weak convergence arguments. We will try to find the weakest possible conditions under which convergence is attained in future work. 

Further work is also needed to explore the eigenfunctions approach. We have sketched a proof of the fact that if the sufficient statistics $c$ of the exponential family $\EF c$ are chosen among the eigenfunctions of the operator ${\cal L}$ associated with the Fokker Planck equation then the Fisher metric projection gives us also the best maximum likelihood estimator of the exact solution. We need to identify SDEs for which the eigenfunction approach is feasible and to study the related approximation. We might be able to show that by including more and more eigenfunctions we could converge in some sense to the true solution.  

In this paper we also tried to clarify how the finite dimensional and infinite dimensional terminology for exponential and mixture spaces are related, since the terms are often used with different meaning in different contexts. This has been clarified to some extent but not completely, and further work remains to be done.

Further work is needed to clarify the $L^2$ direct metric projection in terms of statistical manifolds. The projection based on the $L^2$ structure on densities rather than their square roots, and the related metric, have been used in \cite{armstrongbrigomcss} to work with projection of infinite dimensional evolution equations on finite dimensional mixture families such as the $\MG q$ above. In further work we would like to relate this projection to the statistical and mixture manifold structures based on Orlicz spaces given here rather than in terms of the blunt whole $L^2$ space.  

We would also like to study in the statistical manifold framework the different projections suggested in \cite{armstrongbrigoicms} for evolution equations driven by rough paths. For such equations there is more than one possible projection, depending on the notion of optimality one chooses, which is related to the rough paths properties. This would combine geometry in the space of probability laws with geometry in the state space.

Finally, we would like to examine different measure evolution equations than the few we worked with here. This too will be investigated in further work.

\section{Acknowledgements}
The authors are grateful to the organizers and participants of the conference {\emph{Computational information geometry for image and signal processing}}, held at the ICMS in Edinburgh on  September 21-25 2015. They are also grateful to Frank Nielsen for feedback on this preprint and to an anonynous referee for suggesting investigating the approximation error, as this prompted us to derive the MLE theorem. G. Pistone is supported by deCastro Statistics, Collegio Carlo Alberto, Moncalieri, and he is a member of GNAFA-INDAM.

\bibliographystyle{splncs03}
\bibliography{BP}

\end{document}